\documentclass[a4paper,10pt]{amsart}
\usepackage[arrow,matrix]{xy}
\usepackage{amsfonts,amsmath,amssymb,amscd,bbm,amsthm,mathrsfs,dsfont,bm,enumitem}
\usepackage{extarrows}
\usepackage{latexsym}
\usepackage{graphicx}
\usepackage{color}
\usepackage{xy}
\usepackage{leftidx}
\usepackage{pifont}
\setlist[enumerate]{label=$(\arabic*)$}
\usepackage{cite}
\usepackage{makecell,array}

\allowdisplaybreaks
\newtheorem{theorem}{Theorem}[section]
\newtheorem{lemma}[theorem]{Lemma}
\newtheorem{proposition}[theorem]{Proposition}
\newtheorem{corollary}[theorem]{Corollary}

\theoremstyle{definition}
\newtheorem{definition}[theorem]{Definition}

\newtheorem{remark}[theorem]{Remark}

\numberwithin{equation}{section}
\newcommand{\Z}{\mathbb{Z}}

\newtheorem{example}[theorem]{Example}
\newcommand{\kk}{\mathbbm{k}}
\newcommand{\II}{\mathbbm{I}}
\newcommand{\JJ}{\mathbbm{J}}
 \DeclareMathOperator{\Ker}{Ker}

\DeclareMathOperator{\Hom}{Hom}

\DeclareMathOperator{\uExt}{\underline{Ext}}

\DeclareMathOperator{\uHom}{\underline{Hom}}

\DeclareMathOperator{\gr}{gr}
\DeclareMathOperator{\id}{id}

\DeclareMathOperator{\Gr}{Gr}

\DeclareMathOperator{\umcm}{\underline{mcm}}
\DeclareMathOperator{\gldim}{gldim}

\DeclareMathOperator{\module}{mod}
\def\bt{\begin{theorem}}
\def\et{\end{theorem}}
\def\bl{\begin{lemma}}
\def\el{\end{lemma}}
\def\br{\begin{remark}}
\def\er{\end{remark}}
\def\bc{\begin{corollary}}
\def\ec{\end{corollary}}

\begin{document}

\title{Skew Kn\"orrer's periodicity Theorem
}

\author{Yang Liu}
\address{Liu: Department of Mathematics, Zhejiang Sci-Tech University, Hangzhou 310018, China}
\email{wxly120@gmail.com}

\author{Yuan Shen}
\address{Shen: Department of Mathematics, Zhejiang Sci-Tech University, Hangzhou 310018, China}
\email{yuanshen@zstu.edu.cn}

\author{Xin Wang}
\address{Wang: School of Science, Shandong Jianzhu University, Jinan 250101, China}
\email{wangxin19@sdjzu.edu.cn}

\date{}

\begin{abstract}
In this paper, we introduce a class of twisted matrix algebras of $M_2(E)$ and twisted direct products of $E\times E$ for an algebra $E$.  Let $A$ be a noetherian Koszul Artin-Schelter regular algebra, $z\in A_2$ be a regular central element of $A$ and $B=A_P[y_1,y_2;\sigma]$ be a graded double Ore extension of $A$. We use the Clifford deformation $C_{A^!}(z)$ of Koszul dual $A^!$ to study the noncommutative quadric hypersurface $B/(z+y_1^2+y_2^2)$. We prove that the stable category of graded maximal Cohen-Macaulay
modules over $B/(z+y_1^2+y_2^2)$ is equivalent to certain bounded derived categories, which involve a  twisted matrix algebra of $M_2(C_{A^!}(z))$ or a twisted direct product of $C_{A^!}(z)\times C_{A^!}(z)$ depending on the values of $P$. These results are presented as skew versions of Kn\"orrer's periodicity theorem. Moreover, we show $B/(z+y_1^2+y_2^2)$ may not be a noncommutative graded isolated singularity even if $A/(z)$ is.
\end{abstract}



\maketitle

\section*{Introduction}
Throughout the paper, $\kk$ is an algebraically closed field of characteristic $0$. All vector spaces and algebras are over $\kk$. Unless otherwise stated, the tensor product $\otimes$ means $\otimes_{\kk}$.

Kn\"orrer's periodicity theorem is a powerful tool in the study of the Cohen-Macaulay representation theory and also the singularities of (commutative) quadric hypersurfaces (see \cite{K}). It implies that the stable category $\umcm S$ of maximal Cohen-Macaulay modules over a quadric hypersurface $S=\kk[x_1,\cdots,x_n]/(f)$ is equivalent to $\umcm S^{\#\#}$ where $f\in (x_1,\cdots,x_2)^2$ and $S^{\#\#}=\kk[x_1,\cdots,x_n][u,v]/(f+u^2+v^2)$ is the second double branched cover of $S$, which helps us to reduce the computation of the Cohen-Macaulay representation of a quadric hypersurface.

A noetherian Koszul Artin-Schelter regular algebra $A$ is regarded as a noncommutative polynomial and a coordinate ring of a class of noncommutative projective spaces. Let $z\in A_2$ be a nonzero regular central element of $A$. The quotient algebra $A/(z)$ is an analogue of (commutative) quadric hypersurface, usually called a \emph{noncommutative quadric hypersurface}. There have been many studies in noncommutative quadric hypersurfaces for the past few years (for example, \cite{CKMW,HY,HMY,HU, H,HMM,MU1,MU2,Ue,Ue2,SV}). To study graded Cohen-Macaulay modules of a noncommutative quadric hypersurface $A/(z)$, there are two useful tools. One is a finite dimensional algebra $C(A/(z))$ introduced by Smith and Van den Bergh in \cite{SV}, and the other one is a Clifford deformation $C_{A^!}(z)$ of the Koszul dual $A^!$ of $A$ introduced by He and Ye in \cite{HY}. The Clifford deformation $C_{A^!}(z)$ is a strongly $\Z_2$-graded algebra and $C_{A^!}(z)_0\cong C(A/(z))$ (see \cite{HY}), and so $C_{A^!}(z)$ may provide more information about $A/(z)$.

Similar to the commutative case, the noncommutative Kn\"orrer's periodicity theorem, that is, $\umcm A/(z)$ and $\umcm (A/(z))^{\#\#}$ are equivalent, was proved for a noncommutative quadric hypersurface $A/(z)$ in \cite{CKMW,HY,MU2}. He, Ma and Ye proved a generalized version of Kn\"orrer's periodicity theorem for tensor products of two noncommutative quadric hypersurfaces with the help of Clifford deformations (see \cite[Theorem 2.15]{HMY}).

In noncommutative setting, it is reasonable to consider noncommutative versions of double branched covers and twisted tensor products of noncommutative quadric hypersurfaces.  For example, $T=A[u;\sigma_1][v;\sigma_2]/(z+u^2+v^2)$ defined by two-step iterated Ore extensions for some graded automorphisms $\sigma_1$ and $\sigma_2$ seems to make more sense as a class of second noncommutative double branched covers,  and a natural question is how we can describe $\umcm T$ by $\umcm A/(z)$. If $\sigma_1={\sigma_2}_{\mid A}$ and $\sigma_2(u)=u$ , it is a special case of the noncommutative Kn\"orrer's periodicity theorem given by Mori and Ueyama (see \cite[Theorem 1.3]{MU2}), and this case can also reduce to the second double branched cover $(A/(z))^{\#\#}$ by a Zhang-twist (see \cite{Z}). Nevertheless, we are curious abut how to describe $\umcm T$ in general.

On the other hand, Hu, Matsuno and Mori completely classified noncommutative conics up to isomorphism of noncommutative projective  schemes in \cite{HMM}, that is a class of noncommutative quadric hypersurfaces defined by $3$-dimensional noetherian Koszul Artin-Schelter regular algebras. It inspires us to discuss the classification of noncommutative quadric hypersurfaces obtained by $4$-dimensional noetherian Koszul Artin-Schelter regular algebras. To achieve this goal, it is necessary to compute the Cohen-Macaulay representation of such noncommutative quadric hypersurfaces. Note that Zhang and Zhang introduced an approach called double Ore extension to obtain a classification of  $4$-dimensional noetherian Koszul Artin-Schelter regular algebras (see \cite{ZZ1,ZZ2}). Also, double Ore extensions include two-step iterated Ore extensions. Therefore, it is reasonable to study relations between the graded maximal Cohen-Macaulay modules over noncommutative quadric hypersurfaces $A/(z)$ and $A_P[y_1,y_2;\sigma]/(z+y^2_1+y_2^2)$ where $A_P[y_1,y_2;\sigma]$ is a double Ore extension of $A$.

To our end, we introduce twisted algebra structures on a matrix algebra $M_2(E)$ and a direct product $E\times E$ for an algebra $E$. The twisted matrix algebra of $M_2(E)$ (see Theorem \ref{thm: twisted matrix algebra}) comes from a  twisting system $\Theta$ (see Definition \ref{def: twisting system}) consisting of two linear maps
$
\theta^{(i)}:E\to M_2(E)
$ for $i\in\Z_2$ and an invertible $\Z_2$-graded  basis of $M_2(\kk)$, where $M_2(\kk)_0$ consists of diagonal matrices and $M_2(\kk)_1$ consists of anti-diagonal matrices. Denote the new twisted matrix algebra by ${^\Theta M_2(E)}$, and it is a $\Z_2$-graded algebra. This twisted algebra can be viewed as a generalization of Zhang-twists (see \cite{Z}) for graded algebras, but they have different properties. Furthermore, taking the subalgebra ${^\Theta M_2(E)}_0$ of ${^\Theta M_2(E)}$ provides an approach to obtaining a twisted algebra structure on $E\times E$ (see Corollary \ref{cor: twisted of E oplus E}) and denote the new twisted algebra by ${^{\theta_{\times}}(E\times E)}$ where the twisting system  $\theta_\times$ consists of a linear map $\theta$ from $E$ to $ M_2(E)$ and an invertible basis of $\kk\times \kk$.  If $E$ is a $\Z_2$-graded algebra, then ${^\Theta M_2(E)}$ is $\Z_2\times \Z_2$-graded and also $\Z_2$-graded by taking the total degrees, and ${^{\theta_{\times}}(E\times E)}$ is $\Z_2$-graded.

Let $A$ be a noetherian Koszul Artin-Schelter regular algebra, $z\in A_2$ be a nonzero regular central element and $B=A_P[y_1,y_2;\sigma]$ be a $\Z$-graded  double Ore extension of $A$ where $P=\{p_{11},p_{12}\}\subseteq \kk$. The main object we study is the noncommutative quadric hypersurface $B/(z+y_1^2+y_2^2)$. By the deformation method introduced by He-Ye, $\umcm B/(z+y_1^2+y_2^2)$ is totally determined by the Clifford deformation $C_{B^!}(z+y_1^2+y_2^2)$ of $B^!$ (\cite[Theorem 0.2]{HY}). We aim to describe $C_{B^!}(z+y_1^2+y_2^2)$ by the Clifford deformation $C_{A^!}(z)$ of $A^!$. It should be noted that the semi-trivial extension (see Theorem \ref{thm: semi-trivial extension}) is a useful construction for our goal. To make the noncommutative quadric hypersurface $B/(z+y_1^2+y_2^2)$ well-defined, it forces $p_{12}=\pm 1$. We divide it into two cases.
 
If $p_{12}=1$, $p_{11}$ should be $0$. In this case, the Clifford deformation $C_{B^!}(z+y_1^2+y_2^2)$ is a twisted matrix algebra of $M_2(C_{A^!}(z))$, and $\umcm(B/(z+y_1^2+y_2^2))$ can be described by a semi-trivial extension.

\begin{theorem}\label{main thm: p12=1} (Theorem \ref{thm: +1case, clifford def is twisted matrix algebra} and Theorem \ref{thm: +1case equivalence to Lambda})
Let $B=A_{\{1,0\}}[y_1,y_2;\sigma]$ be a $\Z$-graded double Ore extension of a noetherian Koszul Artin-Schelter regular algebra $A$ with $\deg y_1=\deg y_2=1$ and $z\in A_2$ be a nonzero regular central element of $A$. Suppose $B$ is noetherian and  $z+y_1^2+y_2^2$ is a central element of $B$. Then
\begin{enumerate}
\item there exists a twisting system $\Theta$ of $M_2(C_{A^!}(z))$ such that 
$$C_{B^!}(z+y_1^2+y_2^2)\cong {^\Theta M_2(C_{A^!}(z))},$$ 
as $\Z_2$-graded algebras;

\item there exists a semi-trivial extension $\Lambda=S\ltimes_{\psi}{M(1)}$ of a $\Z_2$-graded subalgebra $S$ of $C_{A^!}(z)$ and a $\Z_2$-graded $S$-bimodule $M$ (a subspace of $C_{A^!}(z)$) through a map $\psi: M(1)\otimes_S M(1)\to S$ such that
	$$\hspace{11mm}\umcm(B/(z+y_1^2+y_2^2))\cong D^b(\gr_{\Z_{2}}{^\Theta M_2(C_{A^!}(z))})\cong D^b(\module {^\Theta M_2(C_{A^!}(z))}_0)\cong D^b(\gr_{\Z_2} \Lambda),$$
as triangulated categories.
\end{enumerate}
\end{theorem}

If $p_{12}=-1$, we can reduce $p_{11}$ to be $0$ in case $p_{11}\neq\pm2\sqrt{-1}$. Then one obtains a twisted direct product $\Gamma={^{\theta_\times}(C_{A^!}(z)\times C_{A^!}(z))}$ for some twisting system $\theta_\times$ and the Clifford deformation $C_{B^!}(z+y_1^2+y_2^2)$ is a semi-trivial extension constructed by $\Gamma$. In particular, the $0$-th component of $C_{B^!}(z+y_1^2+y_2^2)$ is exactly a Zhang-twisted algebra (see \cite{Z}) of $\Gamma$.

\begin{theorem}\label{main thm: p12=-1}(Theorem \ref{thm: -1case equivalence ot semi-trivial extension of Gamma})
Let $B=A_{\{-1,0\}}[y_1,y_2;\sigma]$ be a $\Z$-graded double Ore extension of a noetherian Koszul Artin-Schelter regular algebra $A$ with $\deg y_1=\deg y_2=1$ and $z\in A_2$ be a nonzero regular central element of $A$. Suppose $B$ is noetherian and  $z+y_1^2+y_2^2$ is a central element of $B$. Then 
 \begin{enumerate}
     \item there exist a twisted direct product $\Gamma$ of  $C_{A^!}(z)\times C_{A^!}(z)$ and a semi-trivial extension $\Gamma\ltimes_{\psi} ({_\mu\Gamma(1)})$ of $\Z_2$-graded algebra $\Gamma$ and $\Z_2$-graded $\Gamma$-bimodule ${_\mu\Gamma(1)}$ through some map $\psi$ where $\mu$ is a $\Z_2$-graded automorphism of $\Gamma$ such that 
     $$C_{B^!}(z+y_1^2+y_2^2)\cong\Gamma\ltimes_{\psi} ({_\mu\Gamma(1)})$$ 
     as $\Z_2$-graded algebras.
     \item  there are equivalences of triangulated categories
	$$\umcm(B/(z+y_1^2+y_2^2))\cong D^b\left(\gr_{\Z_{2}}
 \Gamma\ltimes_{\psi} ({_\mu\Gamma(1)})\right)\cong D^b\left(\module 
 {^\nu \Gamma}\right),$$
where $^\nu \Gamma$ is a Zhang-twisted algebra of $\Gamma$ for the left Zhang-twisting system $\nu=\{\nu_0=\id,\nu_1=\mu\}$ of $\Z_2$-graded algebra $\Gamma$.
 \end{enumerate}
\end{theorem}
In particular, Theorem \ref{main thm: p12=1} reduces to the noncommutative Kn\"orrer's periodicity theorem if $\sigma$ is a diagonal map. Theorem \ref{main thm: p12=1} and Theorem \ref{main thm: p12=-1} can be viewed as skew versions of Kn\"orrer's periodicity theorem.

For the exceptional case $p_{12}=-1$ and $p_{11}=\pm2\sqrt{-1}$, we have any noncommutative quadric hypersurface $A_{\{-1,p_{11}\}}[y_1,y_2;\sigma]/(z+y_1^2+y_2^2)$ is not a noncommutative  graded isolated singularity in the sense of \cite{Ue1} (see Proposition \ref{prop: -1case p11=2genhao-1 not isolated singularity}), or whose noncommutative projective scheme is not smooth. Hence, we are also interested in whether $A_{\{\pm 1,0\}}[y_1,y_2;\sigma]/(z+y_1^2+y_2^2)$ is a noncommutative graded isolated singularity in case $A/(z)$ is. We give a negative answer to this question (see Proposition \ref{prop: double ore extension not preserve isolated singulariy}). It should be emphasized that this result is a totally different phenomena to the usual second double branched covers and tensor products of noncommutative quadric hypersurfaces, which preserve the property of noncommutative graded isolated singularity (see \cite[Theorem 0.4]{HY} and \cite[Theorem 2.5]{HMY}).

This paper is organized as follows. In Section 1, we recall some definitions and results, including noncommutative quadric hypersurfaces, Clifford deformations and semi-trivial extensions. In Section 2, we introduce a class of twisted matrix algebras and a class of twisted direct products. In Section 3, we recall the definition of double Ore extensions and show some basic facts about  the Clifford deformations for noncommutative quadric hypersurfaces obtaining by double Ore extensions. We devote Section 4 and Section 5 to proving Theorem \ref{main thm: p12=1} and Theorem \ref{main thm: p12=-1} respectively.

\section{Preliminaries}\label{Section preliminaries}

Let $C=\oplus_{g\in G}C_g$ be a $G$-graded algebra where $G$ is an abelian group. Write  $\Gr_G C$ for the category of $G$-graded right $C$-modules with morphisms of $C$-module homomorphisms preserving grading and $\gr_G C$ for the full subcategory of $\Gr_G C$ consisting of finitely generated $G$-graded right $C$-modules. Let $M$ be a $G$-graded right $C$-module, then $M(h)$ stands for the shifting $G$-graded right $C$-module where $M(h)_g=M_{g+h}$ for any $g\in G$ and  some $h\in G$. For a $G$-graded automorphism $\mu$ of $C$, the twisted $G$-graded right $C$-module $M_\mu$ is the abelian group $M$ with the right $C$-action  $m\cdot c=m\mu(c)$ for any $m\in M,c\in C.$ Similarly, we can define shifting  and twisted $G$-graded left $C$-modules (or, $C$-bimodules).

For $M,M'\in \Gr_G C$, write  $$\uHom_C(M,M')=\oplus_{g\in G}\Hom_{\Gr_G C}(M,M'(g))$$ and $\uExt^i_C$ for the $i$-th derived functor of $\uHom_{C}$. In particular, $\kk$ is always a $G$-graded algebra concentrated in degree 0, and denote the $G$-graded linear dual functor $\uHom_\kk(-,\kk)$ by $(-)^*.$ For any nonzero $k\in\kk$, define a $G$-graded algebra automorphism $\xi_{k}$ of $C$ by 
$\xi_{k}(c)=k^{\deg c}c$
for any homogeneous element $c$ of $C$. 

A set of $G$-graded linear automorphisms $\nu=\{\nu_g\mid g\in G\}$ of $C$ is called a \emph{left Zhang-twisting system} of $C$, if $ \nu_{l}(\nu_h(x)y)=\nu_{hl}(x)\nu_l(y)$
for all $g,h,l\in G$ and $x\in C_g,y\in C_h$. If $\nu=\{\nu_g\mid g\in G\}$ is left Zhang-twisting system of $C$, the graded vector space $C=\oplus_{g\in G}C_g$ with the multiplication $
x\ast y=\nu_h(x)y,
$
for all $x\in C_g,y\in C_h$, is a $G$-graded algebra, called a \emph{Zhang-twisted algebra} of $A$ by $\nu$, denoted by ${^\nu A}$ (see \cite[Section 4]{Z}).

\begin{definition}\label{def: inverse and t-inverse}
Let $C$ be a $G$-graded algebra, and $\sigma=(\sigma_{ij})$ and $\varphi=(\varphi_{ij})$ be two algebra homomorphisms from $C$ to $M_2(C)$ where $\sigma_{ij}$ and $\varphi_{ij}$ are $G$-graded linear transformations of $C$ for any $i,j=1,2$. If
	$$
	\sum_{k=1}^2 \sigma_{ki}\varphi_{kj}=\delta_{ij}\id_C\quad \text{and}\quad 	\sum_{k=1}^2 \varphi_{jk}\sigma_{ik}=\delta_{ij}\id_C,
	$$
where $\delta_{ij}$ is the Kronecker symbol for any $i,j=1,2$, we say $\sigma$ is \emph{invertible}, $\varphi$ is \emph{$t$-invertible}, $\varphi$ is an \emph{inverse} of $\sigma$ and $\sigma$ is a \emph{$t$-inverse} of $\varphi$.

In this case, it is not hard to obtain that the inverse of $\sigma$ and the $t$-inverse of $\varphi$ are both unique.
\end{definition}

\begin{remark}
Definition \ref{def: inverse and t-inverse} is a graded version of  \cite[Definition 1.8]{ZZ1}. Note that the invertible form and $t$-invertible form are not the same. The letter ``$t$'' of $t$-invertible  stands for the word ``transpose'' since the transpose of $\varphi$ has the the same invertible property  of $\sigma$. 
\end{remark}

A non-negative $\Z$-graded algebra $C$ with  $C_0=\kk$ and 
$\dim C_i<\infty$ for any $i\geq 1$ is called a \emph{locally finite connected $\Z$-graded algebra}. A locally finite connected $\Z$-graded algebra $A$ is called a \emph{Koszul algebra} (see \cite{Pr}), if the trivial module $\kk_A$ has a free resolution
$$
\cdots\to P_n\to\cdots \to P_1\to P_0\to \kk\to 0,
$$
where $P_i$ is a $\Z$-graded free right $A$-module generated in degree $i$ for any $i\geq 0$. Then $A\cong T(V)/(R)$ for some finite dimensional vector space $V$ with degree $1$ and some vector space $R\subseteq V\otimes V$. The \emph{Koszul dual} of $A$ is the
quadratic algebra $A^! = T(V^*)/(R^\perp)$, where $V^*$ is the dual vector space of $V$ and $R^\perp\subseteq V^*\otimes V^*$ 
is the orthogonal complement of $R$. The Koszul dual $A^!$ is also Koszul.

\begin{definition}
    A locally finite connected $\Z$-graded algebra $D$ is called an \emph{Artin-Schelter Gorenstein algebra} of injective dimenison $d$ if 
    \begin{enumerate}
        \item $\mathrm{injdim}\, _DD=\mathrm{injdim}\, D_D=d<\infty,$
        \item   $\uExt^i_D(\kk_D,D_D)=\uExt^i_D(_D\kk,_DD)=0$ if $i\neq d$ and $\uExt^d_D(\kk_D,D_D), \uExt^d_D(_D\kk,_DD)$ are both $1$-dimensional.
    \end{enumerate}
In addition, $D$ has finite global dimension, then $D$ is called an \emph{Artin-Schelter regular algebra} of dimension $d$.
\end{definition}

Let $A=T(V)/(R)$ be a noetherian Koszul Artin-Schelter regular algebra of dimension $d$ and $z\in A_2$ be a nonzero regular central element of $A$.
The quotient algebra $S=A/(z)$ is a noncommutative quadric
hypersurface, and it is a noetherian Koszul Artin-Schelter Gorenstein algebra of injective dimension $d-1$  by \cite[Lemma 5.1(a)]{SV}.

The noncommutative projective scheme $\mathrm{qgr}\, S=\gr_{\Z} S/\mathrm{tors}_{\Z} \,S$ where $\mathrm{tors}_{\Z} \,S$ is the full subcategory of $\gr_{\Z}  S$ consisting of finite dimensional $\Z$-graded right $S$-modules. We say $S$ is a \emph{noncommutative graded isolated singularity} if
$\mathrm{qgr}\,S$ has finite global dimension (see \cite{Ue1}).

A finitely generated $\Z$-graded right $S$-module $N$ is called a \emph{maximal Cohen-Macaulay module} if the local cohomology
$$
\mathrm{H}_{\mathfrak{m}}^i(N)=\lim_{n\to \infty}\uExt^i_S(S/S_{\geq n},N)=0, \quad \text{if }i\neq d-1,
$$
where $\mathfrak{m}=S_{\geq 1}$.
Denote the full 
subcategory of $\gr_{\Z} S$ consisting of all the  maximal Cohen-Macaulay modules by $\mathrm{mcm}\,S$ and the stable category of  $\mathrm{mcm}\,S$ by $\umcm S$ which is a triangulated category.

In this paper, we use the Clifford deformation introduced by He-Ye in \cite{HY} to discuss $\umcm(S)$. Recall some results in \cite{HY} here. By \cite[Proposition 5.1]{Sm}, $A^!$ is a Koszul Frobenius algebra. Choose $\hat{z}\in V^{\otimes 2}$ as a lift of $z$.  There is a Clifford map
$$
\begin{array}{cclc}
\theta_z:&R^\perp&\to &k\\
~& f&\mapsto & f(\hat{z}),
\end{array}
$$
satisfying $(\theta_z\otimes 1-1\otimes\theta_z)(V^*\otimes R^\perp\cap R^\perp\otimes V^*)=0$. Note that the map $\theta_z$ is independent of the choice of $\hat{z}$. 
Then we have an algebra 
$$C_{A^!}(z)=T(V^*)/(f-\theta_z(f):f\in R^\perp),$$
which is called a \emph{Clifford deformation} of $A^!$. The $\Z$-graded algebra $T(V^*)$ is also a $\Z_2$-graded algebra if we set the subspace $\oplus_{n\geq 0}(V^*)^{\otimes 2n}$ with degree $0$ and $\oplus_{n\geq 0}(V^*)^{\otimes 2n+1}$  with degree $1$. Then $C_{A^!}(z)$ is also a $\Z_2$-graded algebra induced by the $\Z_2$-grading of $T(V^*)$.

\begin{theorem}\cite[Theorem 0.1 and Theorem 0.2]{HY}\label{thm: properties of clifford def}
Retain the notations above. Then
\begin{enumerate}
	\item $C_{A^!}(z)$ is a strongly $\Z_2$-graded Frobenius algebra;
        \item $S$ is a noncommutative graded isolated singularity if and only if $C_{A^!}(z)$ is a $\Z_2$-graded semisimple algebra;
	\item there are equivalences of triangulated categories
	$$
	\umcm S\cong D^b(\gr_{\Z_2}C_{A^!}(z))\cong D^b(\module C_{A^!}(z)_0).
	$$
\end{enumerate}
\end{theorem}

Let $\sigma=(\sigma_{ij}):A\to M_2(A)$ be an algebra homomorphism such that each $\sigma_{ij}$ is a $\Z$-graded linear transformation of $A$ for any $i,j=1,2$. Write  $\sigma_{\mid_V}=({\sigma_{ij}}_{\mid_V}):V\to M_2(V)$ for the  restriction of $\sigma$ on $A_1=V$. One can extend $\sigma_{\mid_V}$ to an algebra homomorphism 
\begin{equation}\label{eq: lift of sigma}
\widetilde{\sigma}=(\widetilde{\sigma}_{ij}):T(V)\to M_2(T(V))
\end{equation}
such that 
$$
\widetilde{\sigma}(R)\subseteq M_2(R),
$$ 
where each $\widetilde{\sigma}_{ij}$ is a $\Z$-graded linear transformation of $T(V)$ for any $i,j=1,2$. Applying the $\Z$-graded linear dual functor $(-)^*$ to each component of $\widetilde{\sigma}$, one obtains an algebra homomorphism
\begin{equation}\label{eq:  tilde sigma star}
\widetilde{\sigma}^*=(\widetilde{\sigma}^*_{ij}):T(V^*)\to M_2(T(V^*)).
\end{equation}
For any $f\in R^\perp$, $r\in R$ and $i,j=1,2$, we have $\widetilde{\sigma}^*_{ij}(f)(r)=f(\widetilde{\sigma}_{ij}(r))=0.$ It implies
\begin{equation}\label{eq: tilde sigma star presreve R orth}\widetilde{\sigma}^*(R^\perp)\subset M_2(R^\perp).
\end{equation}

Suppose $\sigma_{ij}(z)=\delta_{ij}z$ for any $i,j=1,2$. Then the lift map $\widetilde{\sigma}=(\widetilde{\sigma}_{ij})$ defined in \eqref{eq: lift of sigma} satisfies
$$
\widetilde{\sigma}_{ij}(\hat{z})-\delta_{ij}\hat{z}\in R,
$$
for any $i,j=1,2$. So the dual map $\widetilde{\sigma}^*=(\widetilde{\sigma}^*_{ij}):T(V^*)\to M_2(T(V^*))$ defined in \eqref{eq:  tilde sigma star} satisfies 
$$
\widetilde{\sigma}^*_{ij}(f-f(\widehat{z}))=\widetilde{\sigma}^*_{ij}(f)-\delta_{ij}f(\widehat{z})=\widetilde{\sigma}^*_{ij}(f)-f(\widetilde{\sigma}_{ij}(\hat{z}))=
\widetilde{\sigma}^*_{ij}(f)-\widetilde{\sigma}^*_{ij}(f)(\widehat{z})\in (g-g(\hat{z})\mid g\in R^\perp),
$$
for any $f\in R^\perp$ and $i,j=1,2$, where the first equality holds since $\widetilde{\sigma}^*$ is an algebra homomorphism and $f(\hat{z})\in\kk$, and the final belonging holds by \eqref{eq: tilde sigma star presreve R orth} .  Hence, $\widetilde{\sigma}^*$ induces an algebra homomorphism
\begin{equation}\label{eq: dual map of simga from clifford deformation to its matrix}
    \sigma^!=(\sigma^!_{ij}):C_{A^!}(z)\to M_2(C_{A^!}(z)),
\end{equation}
where each $\sigma^!_{ij}$ is a $\Z_2$-graded linear transformation of $C_{A^!}(z)$ for any $i,j=1,2$.

\begin{lemma}\label{lem: inv of sigma and simga dual}
If $\sigma$ is invertible, then $\sigma^!$ is $t$-invertible. 
\end{lemma}
\begin{proof}
Write $\varphi$ for the inverse of $\sigma$. We have $\delta_{ij}z=\sum_{k=1}^2\varphi_{ik}\sigma_{jk}(z)=\varphi_{ij}(z)$ for $i,j=1,2$. Hence, $\varphi^!$ is well defined by \eqref{eq: dual map of simga from clifford deformation to its matrix} and is the $t$-inverse of $\sigma^!$.
\end{proof}
%
%
%

We close this section by introducing a class of extensions of algebras, which is called \emph{semi-trivial extension} in \cite{Pa}.

\begin{theorem}\label{thm: semi-trivial extension}
	Let $E$ be an algebra and $M$ be an $E$-bimodule.	Let $\psi:M\otimes_EM\to E$ be an $E$-bimodule homomorphism such that $	m\psi(m'\otimes m'')=\psi(m\otimes m')m''$ for any $m,m',m''\in M$. Then 
	the vector space $E\oplus M$ with the multiplication
	\begin{flalign*}
&&	(x,m)\odot (x',m')=(xx'+\psi(m\otimes m'),xm'+mx'),\qquad \forall x,x'\in E,m,m'\in M,\ \ 
	\end{flalign*}
	is an algebra, called a \emph{semi-trivial extension} of $E$ by $M$ through $\psi$, denoted by $E\ltimes_{\psi}M$.
\end{theorem}
\begin{proof}
It is clear $(1,0)$ is the identity. Since $\psi$ is an $E$-bimodule homomorphism and $	m\psi(m'\otimes m'')=\psi(m\otimes m')m''$  for any $m,m',m''\in M$, the associativity of this multiplication holds. 
\end{proof}

\begin{remark}\label{rem: semi-trivial extension}
	\begin{enumerate}
		
	\item The semi-trivial extension $E\ltimes_{\psi}M$ is the trivial extension $E\ltimes M$ if $\psi=0$. 

        \item A semi-trivial extension $E\ltimes_{\psi}M$ can be viewed as a $\Z_2$-graded algebra (or, a superalgebra) if $\deg(x,0)=0,\deg(0,m)=1$ for any $x\in E,m\in M$.
	
		\item\label{rem: G-graded semi-trivial extension} Suppose $E$ is a $G$-graded algebra, $M$ is a $G$-graded $E$-bimodule and $\psi$ is a $G$-graded map where $G$ is an abelian group.
		Then $E\ltimes_{\psi}M$ is a $\Z_2\times G$-graded algebra whose homogeneous space $(E\ltimes_{\psi}M)_{(0,g)}=E_g$ and $(E\ltimes_{\psi}M)_{(1,g)}=M_g$  for any $g\in G$,  and also a $G$-graded algebra by forgetting the first $\Z_2$-degree.

In particular, if $G=\Z_2$, we always view $\Z_2\times \Z_2$-graded algebra $E\ltimes_{\psi}M$ as a $\Z_2$-graded algebra by forgetting the first $\Z_2$-degree in the sequel, that is,
$$
(E\ltimes_{\psi}M)_0=E_0\oplus M_0,\qquad (E\ltimes_{\psi}M)_1=E_1\oplus M_1.
$$
		
	\end{enumerate}	
\end{remark}
\begin{example}\label{ex: twisted E oplus E}
	\begin{enumerate}
		\item Let $E$ be an algebra and $\psi:E\otimes_EE\to E$ be the multiplication of $E$. Then one obtains a semi-trivial extension $E\ltimes_{\psi}E$.
		
		\item  Let $E$ be a $\Z_2$-graded algebra and $\mu$ be a $\Z_2$-graded algebra automorphism of $E$ such that $\mu^2=\id$. Then 
		$$
		\begin{array}{cclc}
			\psi:&{_{\mu}E(1)}\otimes_E{_{\mu}E(1)}&\to& E\\
			 ~&a\otimes b&\mapsto &\mu(a)b,
		\end{array}
		$$
	 	is a $\Z_2$-graded $E$-bimodule homomorphism and satisfies $\psi(a\otimes b)\cdot c=a\psi(b\otimes c)$ for any $a,b,c\in {_{\mu}E(1)}$. 
	 	Hence, we have a semi-trivial extension  $E\ltimes_{\psi}\left({_\mu}E(1)\right)$ which is a  $\Z_2$-graded algebra (see Remark \ref{rem: semi-trivial extension}(3) when $G=\Z_2$). More precisely, $(E\ltimes_{\psi}\left({_\mu}E(1)\right))_0=E_0\oplus E_1$ and $(E\ltimes_{\psi}\left({_\mu}E(1)\right))_1=E_1\oplus E_0$ with multiplication
   $$
   (a,b)\odot (a',b')=(aa'+\mu(b)b',\mu(a)b'+ba'),
   $$
for any homogeneous elements $(a,b),(a',b')\in E\ltimes_{\psi}\left({_\mu}E(1)\right).$
	\end{enumerate}
\end{example}

\section{Two Classes of Twisted Structures}
In this section, we introduce a class of twisting systems for matrix algebra $M_2(E)$ over some algebra $E$ and a class of twisted matrix algebras. As an application, there is a twisted structure on the direct product $E\times E$ a by choosing a subalgebra of a twisted matrix algebra.

\subsection{Twisted matrix algebras}
Let $M_2(\kk)$ be a $\Z_{2}$-graded algebra whose $\Z_2$-grading is given by  
$$
M_2(\kk)_0=
\begin{pmatrix}
	\kk & 0\\
	0 & \kk
\end{pmatrix},\qquad
M_2(\kk)_1=
\begin{pmatrix}
	0 & \kk\\
	\kk & 0
\end{pmatrix}.
$$
Let $\II^{(0)}_{1},\II^{(0)}_{2}$ be invertible and a basis of $M_2(k)_0$, and $\II^{(1)}_{1},\II^{(1)}_{2}$ be invertible and a basis of $M_2(k)_1$. We say  $\II=\{\II^{(i)}_{j}\mid i\in\Z_2,j=1,2\}$ is an   invertible $\Z_2$-graded basis of $M_2(\kk)$. Let $\gamma_1,\gamma_2\in\kk$ such that
$$
\gamma_1\II^{(0)}_1+\gamma_2\II^{(0)}_2=I_2,
$$
where $I_2$ is the $2\times 2$ identity matrix. For any $i,i'\in\Z_2, j=1,2$, write 
\begin{equation}\label{eq: symbol of L}
L^{(ii')}_{j}=
\begin{pmatrix}
	l^{(ii')}_{1j1} &l^{(ii')}_{1j2}\\
	l^{(ii')}_{2j1} &l^{(ii')}_{2j2}	
\end{pmatrix}\in M_2(\kk),
\end{equation}
such that
$$
\II^{(i)}_{j}(\II^{(i')}_{1}, \II^{(i')}_{2})=(\II^{(i+i')}_{1}, \II^{(i+i')}_{2})L^{(ii')}_{j}.
$$
More precisely, for any $i,i'\in\Z_2, j,j'=1,2,$
\begin{equation}\label{eq: multi for II with element in L}
	\II^{(i)}_j\II^{(i')}_{j'}=\sum_{s=1}^2\II^{(i+i')}_sl^{(ii')}_{sjj'}.
\end{equation}
It is not hard to check that each $L^{(ii')}_j$ is invertible for any $i,i'\in \Z_2$ and $j=1,2$.

\begin{lemma}\label{lem: relations for l and gamma}
	\begin{enumerate}
		\item For any $i,i',i''\in \Z_2$ and $j,j',j'',t=1,2$,
		$$
		\sum_{s=1}^2l^{(i,i'+i'')}_{tjs}l^{(i'i'')}_{sj'j''}=\sum_{s=1}^2l^{(i+i',i'')}_{tsj''}l^{(ii')}_{sjj'}.
		$$
		
		\item For any $i\in\Z_2$ and $s,j=1,2$, 	$\sum_{t=1}^2l^{(i0)}_{sjt}\gamma_t=\delta_{sj}$.
		
		\item For any $i\in\Z_2$ and $s,t=1,2$, $\sum_{j=1}^2l^{(0i)}_{sjt}\gamma_j=\delta_{st}$.
	\end{enumerate}
	
\end{lemma}
\begin{proof}
	(1) It follows by the associativity of  $\II^{(i)}_{j}\II^{(i)}_{j'}\II^{(i'')}_{j''}$.
	
	(2) By \eqref{eq: multi for II with element in L} with $i'=0$, 
	$
	\sum_{s=1}^2\left(\sum_{t=1}^2l^{(i0)}_{sjt}\gamma_t\right)\II^{(i)}_s=\sum_{t=1}^2\II^{(i)}_j\II^{(0)}_t\gamma_t=\II^{(i)}_j.
	$ 
	
	(3) By \eqref{eq: multi for II with element in L} with $i=0$,
	$
	\sum_{s=1}^2\left(\sum_{j=1}^2\gamma_jl^{(0i')}_{sjt}\right)\II^{(i')}_s=\sum_{j=1}^2\gamma_j\II^{(0)}_j\II^{(i')}_t=\II^{(i')}_t.
	$
\end{proof}

Let $E$ be an algebra. The basis $\II$ of $M_2(\kk)$ provides a decomposition for $\Z_2$-graded vector space $M_2(E)$ as follows, for any $i\in\Z_2$
\begin{align*}
	M_2(E)_{i}=\left\{\II^{(i)}_{1}x\mid x\in E\right\}\oplus\left \{\II^{(i)}_{2}x\mid x\in E\right\}.
\end{align*}

To introduce a twisted system of the matrix algebra $M_2(E)$, we consider a class of linear maps from $E$ to the matrix algebra $M_2(E)$.
\begin{proposition}\label{prop: equivalent condition for theta}
	Let $\theta^{(i)}=(\theta^{(i)}_{jj'}): E\to M_2(E)
	$ be a linear map for $i\in \Z_2$. Then the following are equivalent:
	\begin{enumerate}
		\item 	For any $x,x'\in E$, $i,i',i''\in \Z_2$ and $j,j',j'',v=1,2$,
		\begin{equation*}
			\sum_{s,t,u=1}^2 l^{(i+i',i'')}_{vtu}l^{(ii')}_{tjs}\theta^{(i'')}_{uj''}\left(\theta^{(i')}_{sj'}(x)x'\right)=\sum_{s,t,u=1}^2 l^{(i,i'+i'')}_{vjs}
			l^{(i'i'')}_{tj'u}\theta^{(i'+i'')}_{st}(x)\theta^{(i'')}_{uj''}(x').
		\end{equation*}
		
		\item 	For any $x,x'\in E$, $i,i',i''\in \Z_2$ and $j,j',j'',v=1,2$, 
		$$
		\sum_{s,t,u=1}^2 l^{(i,i'+i'')}_{vjt}l^{(i'i'')}_{tsu }\theta^{(i'')}_{uj''}\left(\theta^{(i')}_{sj'}(x)x'\right)=\sum_{s,t,u=1}^2 l^{(i,i'+i'')}_{vjs}
		l^{(i'i'')}_{tj'u}\theta^{(i'+i'')}_{st}(x)\theta^{(i'')}_{uj''}(x').
		$$
		
		\item For any $x,x'\in E$, $i',i''\in \Z_2$ and $j',j'',p=1,2,$
		\begin{equation*}
			\sum_{s,u=1}^2 l^{(i'i'')}_{psu }\theta^{(i'')}_{uj''}\left(\theta^{(i')}_{sj'}(x)x'\right)=\sum_{t,u=1}^2 
			l^{(i'i'')}_{tj'u}\theta^{(i'+i'')}_{pt}(x)\theta^{(i'')}_{uj''}(x').
		\end{equation*}
	\end{enumerate}
\end{proposition}
\begin{proof}
	(1)$\Leftrightarrow$(2) follows by Lemma \ref{lem: relations for l and gamma}(1).
	
	(2)$\Rightarrow$(3) Write $(L^{(ii')}_j)^{-1}=(g^{(ii')}_{j'jj''})$ for any $i,i'\in\Z_2$ and $j,j',j''=1,2.$ Multiplying the equation in (2) by $g^{(i,i'+i'')}_{pjv}$ on both sides and taking the sum over $v$ from $1$ to $2$ for any $p=1,2$, one obtains (3).
	
	(3)$\Rightarrow$(2)  Multiply the equation in (3) by  $l^{(i,i'+i'')}_{vjp}$ for any $j,v=1,2$ on both sides and take the sum over $p$ from $1$ to $2$.
\end{proof}

\begin{lemma}\label{lem: properitis of theta}
	Let $\theta^{(i)}=(\theta^{(i)}_{jj'}): E\to M_2(E)
$ be a linear map for $i\in \Z_2$. Suppose $\theta^{(0)}$ and $\theta^{(1)}$ are $t$-invertible with the $t$-inverses $\varphi^{(0)}$ and $\varphi^{(1)}$ respectively, and satisfy the equivalent conditions of Proposition \ref{prop: equivalent condition for theta}. Then
		$$
		\sum_{u,j=1}^2 l^{(ii')}_{pqu }\theta^{(i')}_{uj}\left(x\varphi^{(i')}_{rj}(x')\right)=\sum_{t,j=1}^2 
		l^{(ii')}_{tjr}\theta^{(i+i')}_{pt}(\varphi^{(i)}_{qj}(x))x',
		$$
		for any $x,x'\in E$, $i,i'\in \Z_2$ and $p,q,r=1,2$.	
	\end{lemma}
	
	\begin{proof}
		Rewrite the indexes of Proposition \ref{prop: equivalent condition for theta}(3):
		\begin{equation*}
			\sum_{s,u=1}^2 l^{(ii')}_{psu }\theta^{(i')}_{uj'}\left(\theta^{(i)}_{sj}(x)x'\right)=\sum_{t,u=1}^2 
			l^{(ii')}_{tju}\theta^{(i+i')}_{pt}(x)\theta^{(i')}_{uj'}(x'),
		\end{equation*}
		for any $x,x'\in E$, $i,i'\in \Z_2$ and $j,j',p=1,2$.
		
		Replacing $x$ and $x'$ by $\varphi^{(i)}_{qj}(x)$ and $\varphi^{(i')}_{rj'}(x')$ respectively and taking the sum over $j,j'$ both from $1$ to $2$,  the result follows.
	\end{proof}
	\begin{lemma}\label{lem: prperties of theta(1) and varphi(1)}
		Retain the assumptions of Lemma \ref{lem: properitis of theta}. Then 
		$\theta^{(0)}(1)\in GL_2(\kk)$ if and only if $\theta^{(1)}(1)\in GL_2(\kk)$. In this case, 	$\varphi^{(i)}(1)\in GL_2(\kk)$ for $i\in\Z_2$. 
	\end{lemma}
	\begin{proof}
		Suppose $\theta^{(0)}(1)\in GL_2(\kk)$,
		$(\varphi^{(0)}(1))^T\theta^{(0)}(1)=((\varphi^{(0)})^T\theta^{(0)})(1)=I_2.$ So $\varphi^{(0)}(1)\in GL_2(\kk)$.
		
		Taking $x=x'=1$ and $i=0,i'=1$ in Lemma \ref{lem: properitis of theta}, we have 
		$$
		l^{(01)}_{pqr}=\sum_{t,j=1}^2 
		l^{(01)}_{tjr}\theta^{(1)}_{pt}(\varphi^{(0)}_{qj}(1))=\sum_{t,j=1}^2 
		l^{(01)}_{tjr}\theta^{(1)}_{pt}(1)\varphi^{(0)}_{qj}(1),
		$$
		for any $p,q,r=1,2$. Write those equations in the form of matrices:
		$$
		\overline{L}^{(01)}_{1,2;r}=\theta^{(1)}(1)\overline{L}^{(01)}_{1,2;r}(\varphi^{(0)}(1))^T,
		$$
		where $\overline{L}^{(01)}_{1,2;r}=\begin{pmatrix}
			l^{(01)}_{11r}&	l^{(01)}_{12r}\\
			l^{(01)}_{21r}&	l^{(01)}_{22r}
		\end{pmatrix}$ for any $r=1,2$. Note that for any $r=1,2$, $\{\II^{(0)}_1\II^{(1)}_r,\II^{(0)}_2\II^{(1)}_r\}$ is also a basis of $M_2(\kk)_1$ and
		$$
		(\II^{(0)}_1\II^{(1)}_r,\II^{(0)}_2\II^{(1)}_r)=(\II^{(1)}_1,\II^{(1)}_2)\overline{L}^{(01)}_{1,2;r}.
		$$
		So $\overline{L}^{(01)}_{1,2;r}$ is invertible and
		$$
		\theta^{(1)}(1)=\overline{L}^{(01)}_{1,2;r}\theta^{(0)}(1)\left(\overline{L}^{(01)}_{1,2;r}\right)^{-1}\in GL_2(\kk).
		$$
	 Similar to the invertibility of $\varphi^{(0)}(1)$, one obtains that $\varphi^{(1)}(1)\in GL_2(\kk)$.
		
		Similarly, if $\theta^{(1)}(1)\in GL_2(\kk)$, then $\varphi^{(1)}(1)\in GL_2(\kk)$, and 
		$$
		\theta^{(0)}(1)=\overline{L}^{(11)}_{1,2;r}\theta^{(1)}(1)\left(\overline{L}^{(11)}_{1,2;r}\right)^{-1}\in GL_2(\kk),
		$$
		where 
		$\overline{L}^{(11)}_{1,2;r}=\begin{pmatrix}
			l^{(11)}_{11r}&	l^{(11)}_{12r}\\
			l^{(11)}_{21r}&	l^{(11)}_{22r}
		\end{pmatrix}$ is invertible for $r=1,2$.
	\end{proof}

	\begin{lemma}\label{lem: relation between gamma and theta(1) varphi(1)} Retain the assumptions of Lemma \ref{lem: properitis of theta}. Suppose $\theta^{(0)}(1)\in GL_2(\kk)$ (or $\theta^{(1)}(1)\in GL_2(\kk)$).
		\begin{enumerate}
			
			\item For any $i,i'\in \Z_2$ and $q,r,s=1,2$, then 
			$
			\sum_{p=1}^2l^{(ii')}_{pqr }\varphi^{(i+i')}_{ps}(1)=\sum_{j=1}^2 
			l^{(ii')}_{sjr}\varphi^{(i)}_{qj}(1).
			$

			\item For any $x\in E$ and $v=1,2$, then
			$
			\sum_{r,j=1}^2  \theta^{(0)}_{vj}\left(x\varphi^{(0)}_{rj}(1)\right)\gamma_r=x\gamma_v.
			$
			\item For any $i\in \Z_2$ and $q,s=1,2$, then
			$\sum_{t,j,u=1}^2 \gamma_ul^{(0i)}_{sjt}\varphi^{(0)}_{uj}(1)\theta^{(i)}_{tq}(1)=\delta_{sq}.$
		\end{enumerate}
		
	\end{lemma}
	\begin{proof}
		(1) Taking $x=1, x'=\varphi^{(i+i')}_{ps}(1)$ in Lemma \ref{lem: properitis of theta} and by Lemma \ref{lem: prperties of theta(1) and varphi(1)}, the result follows.
		
		(2) Taking $i'=0$ and $x'=1$ in Lemma \ref{lem: properitis of theta}, one obtains that
		$$
		\sum_{u,j=1}^2 l^{(i0)}_{pqu }\theta^{(0)}_{uj}\left(x\varphi^{(0)}_{rj}(1)\right)=\sum_{t,j=1}^2 
		l^{(i0)}_{tjr}\theta^{(i)}_{pt}(\varphi^{(i)}_{qj}(x)),
		$$
		for any $p,q,r=1,2$. 
		By Lemma \ref{lem: relations for l and gamma}(2), we have
		$$
		\sum_{r,u,j=1}^2 l^{(i0)}_{pqu }\theta^{(0)}_{uj}\left(x\varphi^{(0)}_{rj}(1)\right)\gamma_r=\sum_{r,t,j=1}^2 
		l^{(i0)}_{tjr}\theta^{(i)}_{pt}(\varphi^{(i)}_{qj}(x))\gamma_r=\sum_{j=1}^2 
		\theta^{(i)}_{pj}(\varphi^{(i)}_{qj}(x))=\delta_{pq}x,
		$$
		and then multiply the equation above by $g^{(i0)}_{vqp}$ which is the $(v,p)$-th element of $(L^{(i0)}_q)^{-1}$ for any $v,p,q=1,2$. The result follows.
		
		(3) By Lemma \ref{lem: relations for l and gamma}(3), Lemma \ref{lem: prperties of theta(1) and varphi(1)} and taking $i=0$ in (1), one obtains 
		\begin{equation*}\sum_{t,j,u=1}^2 \gamma_ul^{(0i)}_{sjt}\varphi^{(0)}_{uj}(1)\theta^{(i)}_{tq}(1)=
			\sum_{p,t,u=1}^2 \gamma_ul^{(0i)}_{put}\varphi^{(i)}_{ps}(1)\theta^{(i)}_{tq}(1)=\sum_{t=1}^2 \varphi^{(i)}_{ts}(1)\theta^{(i)}_{tq}(1)=\delta_{sq}.\qedhere
		\end{equation*}
	\end{proof}

  \begin{definition}\label{def: twisting system}
	Let $\theta^{(i)}=(\theta^{(i)}_{jj'}): E\to M_2(E)
		$ be a $t$-invertible linear map for $i\in\Z_2$ and $\II=\{\II^{(i)}_j\mid i\in\Z_2,j=1,2\}$ be an invertible $\Z_2$-graded basis of $M_2(\kk)$. The pair $\Theta=(\{\theta^{(i)}\}_{i\in\Z_2},\II)$ is called a \emph{twisting system} of $M_2(E)$ if
		\begin{enumerate}
			\item[(1)] $\theta^{(0)},\theta^{(1)}$ satisfy equivalent conditions of Proposition \ref{prop: equivalent condition for theta}, 

			\item[(2)]  	$
			\theta^{(1)}(1)\in GL_2(\kk) $ (or equivalently, $
			\theta^{(0)}(1)\in GL_2(\kk)
			$).
		\end{enumerate}

	\end{definition}

 	\begin{theorem}\label{thm: twisted matrix algebra}
		Let $\Theta=(\{\theta^{(i)}\}_{i\in\Z_2},\II)$ be a twisting system of $M_2(E)$. Then the $\Z_2$-graded vector space $M_2(E)$ with the multiplication as follows, for any $x,x'\in E, i,i'\in \Z_2$ and $j,j'=1,2$, 
		\begin{align*}
		\II^{(i)}_{j}x\star \II^{(i')}_{j'}x'=\sum_{s=1}^2\II^{(i)}_{j}\II^{(i')}_{s}\theta^{(i')}_{sj'}(x)x'=\sum_{s,t=1}^2\II^{(i+i')}_{t}l^{(ii')}_{tjs}\theta^{(i')}_{sj'}(x)x',
		\end{align*}
is a $\Z_2$-graded algebra with the identity $\sum_{j,s=1}^2\gamma_s(\theta^{(0)})^{-1}_{sj}(1)\II^{(0)}_j$, called this new algebra the \emph{twisted matrix algebra of $M_2(E)$ by $\Theta$}, denoted by ${^{\Theta}M_2(E)}$.
	\end{theorem}
	\begin{proof}
		It is straightforward to check the multiplication $\star$ satisfies the associative law by Proposition \ref{prop: equivalent condition for theta}(1).
		
		Write $\varphi^{(0)}$ for the $t$-inverse of $\theta^{(0)}$ and  $\II^{(0)}=\sum_{j,s=1}^2\gamma_s\varphi^{(0)}_{sj}(1)\II^{(0)}_j$. 
		By Lemma \ref{lem: relation between gamma and theta(1) varphi(1)}(2) and Lemma \ref{lem: relations for l and gamma}(2), for any $x\in E$, $i\in \Z_2$ and $j=1,2$,
		\begin{align*}
			\II^{(i)}_jx\star \II^{(0)}=\sum_{s,t,j',u=1}^2\II^{(i)}_tl^{(i0)}_{tjs}\theta^{(0)}_{sj'}(x)\gamma_u\varphi^{(0)}_{uj'}(1)=\sum_{s,t=1}^2\II^{(i)}_tl^{(i0)}_{tjs}\gamma_sx=\II^{(i)}_jx.
		\end{align*}
		On the other hand, by Lemma \ref{lem: relation between gamma and theta(1) varphi(1)}(3), for any  $x\in E$, $i\in \Z_2$ and $j=1,2$,
		\begin{align*}
			\II^{(0)}\star \II^{(i)}_jx =\sum_{s,t,j',u=1}^2\II_t^{(i)}\gamma_u\varphi^{(0)}_{uj'}(1)l^{(0i)}_{tj's}\theta^{(i)}_{sj}(1)x=\II^{(i)}_jx.
		\end{align*}
		Hence, $\II^{(0)}$ is the identity.
	\end{proof}

The last theorem provides a twisted algebra structure on the matrix algebra $M_2(E)$, but the identity seems complicated. It is natural to find a twisting system such that the identity matrix is the identity of the associated twisted matrix algebra. We have an analogous result of \cite[Proposition 2.4]{Z}.

	\begin{theorem}\label{thm: twisted matrix with identity I_2}
		Let $\Theta=(\{\theta^{(i)}\}_{i\in\Z_2},\II)$   be a twisting system of $M_2(E)$. Let $\Upsilon=(\upsilon^{(0)},\upsilon^{(1)})$ be a pair of linear maps from $E$ to $M_2(E)$ satisfying
		\begin{align*}
			\upsilon^{(i)}(x)=\theta^{(i)}(x)(\varphi^{(i)}(1))^{T},	
		\end{align*}
		where $\varphi^{(i)}$ is the $t$-inverse of $\theta^{(i)}$ for any $i\in\Z_2$ and $x\in E$. Then
		\begin{enumerate}
			\item  $\upsilon ^{(i)}$ is $t$-invertible with the $t$-inverse map $E$ from $M_2(E)$ mapping $x$ to 	$\varphi^{(i)}(x)(\theta^{(i)}(1))^{T}$ for any $i\in\Z_2$;
			\item  $\Upsilon=(\{\upsilon^{(i)}\}_{i\in\Z_2},\II)$ is a twisting system of $M_2(E)$;
   \item the identity of the  twisted matrix algebra  ${^\Upsilon M_2(E)}$ is the identity matrix $I_2$;

   \item ${^\Theta M_2(E)}\cong {^\Upsilon M_2(E)}$ as $\Z_2$-graded algebras.
		\end{enumerate}
	\end{theorem}
	\begin{proof}
		(1) It is straightforward. 
		
		(2) It is clear that $\upsilon^{(i)}(1)=\theta^{(i)}(1)(\varphi^{(i)}(1))^T=I_2\in GL_2(\kk)$ for any $i\in \Z_2$ by Lemma \ref{lem: prperties of theta(1) and varphi(1)}. For any $x,x'\in E$, $i,i',i''\in\Z_2$, and $j,j',j''=1,2$,
		\begin{align*}
			\sum_{s,t,u=1}^2 l^{(i+i',i'')}_{vtu}l^{(ii')}_{tjs}\upsilon^{(i'')}_{uj''}\left(\upsilon^{(i')}_{sj'}(x)x'\right)&=\sum_{s,t,u,p,q=1}^2 l^{(i+i',i'')}_{vtu}l^{(ii')}_{tjs}\theta^{(i'')}_{up}\left(\theta^{(i')}_{sq}(x)x'\right)\varphi_{j'q}^{(i')}(1)\varphi_{j''p}^{(i'')}(1)\\
			&=\sum_{s,t,u,p,q=1}^2 l^{(i,i'+i'')}_{vjs}
			l^{(i'i'')}_{tqu}\theta^{(i'+i'')}_{st}(x)\theta^{(i'')}_{up}(x')\varphi_{j'q}^{(i')}(1)\varphi_{j''p}^{(i'')}(1)\\
			&=\sum_{s,t,u,p,q=1}^2 l^{(i,i'+i'')}_{vjs}
			l^{(i'i'')}_{qj'u}\theta^{(i'+i'')}_{st}(x)\theta^{(i'')}_{up}(x')\varphi_{qt}^{(i'+i'')}(1)\varphi_{j''p}^{(i'')}(1)\\
			&=\sum_{s,u,q=1}^2 l^{(i,i'+i'')}_{vjs}
			l^{(i'i'')}_{qj'u}\upsilon^{(i'+i'')}_{sq}(x)\upsilon^{(i'')}_{uj''}(x'),
		\end{align*}
		where the second equality follows by the Proposition \ref{prop: equivalent condition for theta}(1) and the third equality follows by Lemma \ref{lem: relation between gamma and theta(1) varphi(1)}(1).
		
		(3)	By Theorem \ref{thm: twisted matrix algebra}, the identity of  ${^\Upsilon M_2(E)}$ is $$\sum_{j,s,t=1}^2\gamma_s\varphi^{(0)}_{st}(1)\theta^{(0)}_{jt}(1)\II^{(0)}_j=\sum_{j,s,t=1}^2\gamma_s\theta^{(0)}_{jt}(\varphi^{(0)}_{st}(1))\II^{(0)}_j=\sum_{j=1}^2\gamma_j\II^{(0)}_j=I_2.$$
			
	(4)		Define two $\Z_{2}$-graded linear maps:
			$$
			\begin{array}{cclc}
				F: & {^\Theta M_2(E)} &\to&{^\Upsilon M_2(E)}\\
				~ &\II^{(i)}_{j}x&\mapsto &\sum_{s=1}^2\II^{(i)}_{s}x\theta^{(i)}_{sj}(1),
			\end{array}
			\qquad
			\begin{array}{cclc}
				G: & {^\Upsilon M_2(E)} &\to&{^\Theta M_2(E)}\\
				~ &\II^{(i)}_{j}x&\mapsto &\sum_{s=1}^2\II^{(i)}_{s}x\varphi^{(i)}_{js}(1).
			\end{array}
			$$
			It is straightforward to verify that $F$ and $G$ are the inverse to each other. Moreover, $F$ is an algebra homomorphism, since for any $i,i'\in\Z_2$ and $j=1,2$
			\begin{align*}
				F\left(\II^{(i)}_{j}x\right)\star F\left(\II^{(i')}_{j'}x'\right)&=\sum_{p=1}^2\II^{(i)}_{p}x\theta^{(i)}_{pj}(1) \star \sum_{q=1}^2\II^{(i')}_{q}x'\theta^{(i')}_{qj'}(1)\\
				&=\sum_{p,q,s,t=1}^2\II^{(i+i')}_{t}l^{(ii')}_{tps}\upsilon^{(i')}_{sq}(x\theta^{(i)}_{pj}(1))\theta^{(i')}_{qj'}(1)x'   \\
				&=\sum_{p,q,s,t,u=1}^2\II^{(i+i')}_{t}\theta^{(i)}_{pj}(1) \theta^{(i')}_{qj'}(1) l^{(ii')}_{tps}\theta^{(i')}_{su}(x\varphi^{(i')}_{qu}(1))x'\\
				&=\sum_{p,t=1}^2\II^{(i+i')}_{t}\theta^{(i)}_{pj}(1)\left(\sum_{q,v,w=1}^2l^{(ii')}_{wvq}\theta^{(i+i')}_{tw}(\varphi^{(i)}_{pv}(x))\theta^{(i')}_{qj'}(1)\right)x'\\
				&=\sum_{p,t=1}^2\II^{(i+i')}_{t}\theta^{(i)}_{pj}(1)\left(\sum_{q,u,j''=1}^2l^{(ii')}_{tpu}\theta^{(i')}_{uj''}\varphi^{(i')}_{qj''}(x)(\theta^{(i')}_{qj'}(1))\right)x'\\
				&=\sum_{p,u,t=1}^2\II^{(i+i')}_{t}\theta^{(i)}_{pj}(1)l^{(ii')}_{tpu}\theta^{(i')}_{uj'}(x)x'\\
				&=\sum_{u,s,t=1}^2\II^{(i+i')}_{t}l^{(ii')}_{sju}\theta^{(i+i')}_{ts}(1)\theta^{(i')}_{uj'}(x)x'\\   
				&=F\left(\II^{(i)}_{j}x \star \II^{(i')}_{j'}x'\right),
			\end{align*}
			where the fourth and the fifth equalities are by Lemma \ref{lem: properitis of theta}, and the seventh equality is by Proposition \ref{prop: equivalent condition for theta}(c).
		\end{proof}

Before that, the twisting systems on matrix algebras depend on an invertible $\Z_2$-graded basis of $M_2(\kk)$. In fact, we can fix an invertible $\Z_2$-graded basis of $M_2(\kk)$ to obtain twisted matrix algebras.

\begin{theorem}\label{thm: twisted matrix algebra with a fixed basis}
Let $\Theta=(\{\theta^{(i)}\}_{i\in\Z_2},\II)$  be a twisting system of $M_2(E)$. Suppose $\JJ=\{\JJ^{(i)}_j\mid i\in\Z_2,j=1,2\}$ be an  invertible $\Z_2$-graded basis of $M_2(\kk)$. Then there exists a twisting system $\Omega=(\{\omega^{(i)}\}_{i\in\Z_2},\JJ)$ of $M_2(E)$  such that ${^\Theta M_2(E)}\cong {^\Omega M_2(E)}$ as $\Z_2$-graded algebras. 	
\end{theorem}
\begin{proof}
There exist two invertible matrices $U^{(0)}$ and $U^{(1)}$ of $M_2(\kk)$ such that 
$$
(\II^{(i)}_1,\II^{(i)}_2)=(\JJ^{(i)}_1,\JJ^{(i)}_2)U^{(i)},
$$
for any $i\in\Z_2$. Write $\Omega=(\omega^{(0)},\omega^{(1)})$, where
$$
\begin{array}{cclc}
    \omega^{(i)}:& E&\to &M_2(E)  \\
     ~& x&\mapsto &U^{(i)}\theta^{(i)}(x)(U^{(i)})^{-1},
\end{array}
$$
for any $i\in\Z_2$. It is straightforward to check that $\Omega=(\{\omega^{(i)}\}_{i\in\Z_2},\JJ)$ is a twisting system of $M_2(E)$ by Definition \ref{def: twisting system}.

It is also not hard to obtain that the isomorphism between ${^\Theta M_2(E)}$ and ${^\Omega M_2(E)}$ as $\Z_2$-graded vector spaces defined by $U^{(0)}$ and $U^{(1)}$ is an algebra isomorphism.
\end{proof}

The following result follows by Theorem \ref{thm: twisted matrix with identity I_2} and Theorem \ref{thm: twisted matrix algebra with a fixed basis} immediately.

\begin{corollary}
Let $\Theta=(\{\theta^{(i)}\}_{i\in\Z_2},\II)$  be a twisting system of $M_2(E)$. Suppose $\JJ=\{\JJ^{(i)}_j\mid i\in\Z_2,j=1,2\}$ be an  invertible $\Z_2$-graded basis of $M_2(\kk)$. Then there exists a twisting system $\Omega=(\{\omega^{(i)}\}_{i\in\Z_2},\JJ)$ of $M_2(E)$ such that ${^\Theta M_2(E)}\cong {^\Omega M_2(E)}$ as $\Z_2$-graded algebras and the identity of ${^\Omega M_2(E)}$ is $I_2$.
\end{corollary}

This corollary implies that we can always fix an invertible $\Z_2$-graded basis of $M_2(\kk)$ and assume the identity of a twisted matrix algebra is $I_2$.

We turn to graded cases. Let $E$ be a $G$-graded algebra where $G$ is an abelian group. Suppose $\Theta=(\{\theta^{(i)}\}_{i\in\Z_2},\II)$ is a twisting system of $M_2(E)$ where each  $\theta^{(i)}_{jj'}$ is a $G$-graded linear map for any $i\in \Z_2$ and $j,j'=1,2$.  Then the twisted matrix algebra ${^{\Theta}M_2(E)}$ is a $\Z_2\times G$-graded algebra, where each component is
\begin{equation*}
    {^{\Theta}M_2(E)}_{(i,g)}=\left\{\II^{(i)}_{1}x\mid x\in E_g\right\}\oplus\left \{\II^{(i)}_{2}x\mid x\in E_g\right\},
\end{equation*}
for any $i\in \Z_2$ and $g\in G$.

\begin{remark}\label{rem: Z_2-grading of tiwsted matrix algebra} In particular, if $E$ is a $\Z_2$-graded algebra, then $\Z_2\times \Z_2$-graded algebra ${^{\Theta}M_2(E)}$ has the following two approaches to  becoming a $\Z_2$-graded algebra.
\begin{enumerate}
\item By taking the total degrees, ${^{\Theta}M_2(E)}$ becomes a $\Z_2$-graded algebra where
\begin{equation*}
	{^{\Theta}M_2(E)}_0=
	\begin{pmatrix}
		E_0 & E_1\\
		E_1 & E_0
	\end{pmatrix},\qquad
	{^{\Theta}M_2(E)}_1=
	\begin{pmatrix}
		E_1 & E_0\\
		E_0 & E_1
	\end{pmatrix}.
\end{equation*}

\item By forgetting the first degree, ${^{\Theta}M_2(E)}$ becomes a $\Z_2$-graded algebra where
\begin{equation*}
	{^{\Theta}M_2(E)}_0=
	\begin{pmatrix}
		E_0 & E_0\\
		E_0 & E_0
	\end{pmatrix},\qquad
	{^{\Theta}M_2(E)}_1=
	\begin{pmatrix}
		E_1 & E_1\\
		E_1 & E_1
	\end{pmatrix}.
\end{equation*}
\end{enumerate}
These are the only two good $\Z_2$-grading on $M_2(E)$ (see \cite{NV}).
\end{remark}

\begin{remark}
Let $E$ be a $\Z_2$-graded algebra. By taking the total degrees, ${^\Theta M_2(E)}$ and ${^{\Theta'} M_2(E)}$ become $\Z_2$-graded  for some twisting systems $\Theta$ and $\Theta'$ of $M_2(E)$. In general, $\gr_{\Z_2}{^\Theta M_2(E)}$ and $\gr_{\Z_2}{^{\Theta'} M_2(E)}$ are not equivalent and even not derived equivalent (see Example \ref{ex: +1case two setp Ore extensions}). It is different from the one of Zhang-twists (see \cite[Theorem 3.1]{Z}).
\end{remark}

\subsection{Twisted direct products} Let $E$ be an algebra and $\Theta=(\{\theta^{(i)}\}_{i\in\Z_2},\II)$  be a twisting system of $M_2(E)$. The $0$-th homogeneous space
$$
\begin{pmatrix}
	E & 0\\
	0 & E
\end{pmatrix}
$$
of $\Z_2$-graded algebra $
{^\Theta M_2(E)}
$ provides a subalgebra of ${^\Theta M_2(E)}$, 
which is isomorphic to $E\times E$ as vector spaces. It implies that there is a kind of twisted algebra structure on  $E\times E$ induced by the one of $M_2(E)$. It is analogous to Definition \ref{def: twisting system} and Theorem \ref{thm: twisted matrix algebra}, we just list them.

Let $\varepsilon_1=(a_1,a_2),\varepsilon_2=(b_1,b_2)$ be an invertible $\kk$-basis of $\kk\times \kk$ where $a_1a_2b_1b_2\neq0$, and $\{l_{p;jj'}\mid p,j,j'=1,2\}\subseteq \kk$ such that
$
\varepsilon_j\varepsilon_j'=\sum_{p=1}^2 \varepsilon_pl_{p;jj'}
$ for any $j,j'=1,2$.

\begin{definition}\label{def: twisting system of direct product}
	Let $\theta=(\theta_{jj'}):E\to M_2(E)$ be a $t$-invertible linear map and $\varepsilon=\{\varepsilon_1,\varepsilon_2\}$ be an invertible basis of $\kk\times \kk$. We say $\theta_\times=(\theta,\varepsilon)$ is a \emph{twisting system} of $E\times E$, if   
	$\theta(1)\in GL_2(\kk)$ and 
	\begin{equation*}
		\sum_{s,u=1}^2 l_{p;su }\theta_{uj'}\left(\theta_{sj}(x)x'\right)=\sum_{t,u=1}^2 
		l_{t;ju}\theta_{pt}(x)\theta_{uj'}(x'),
	\end{equation*}
	for any $x,x'\in E$ and $j,j',p=1,2.$
\end{definition}

\begin{corollary}\label{cor: twisted of E oplus E}
	Let $\theta_\times=(\theta,\varepsilon)$ be a twisting system of $E\times E$. Then the vector space $E\times E$ with the multiplication
\begin{flalign*}
&&\varepsilon_{j}x\star \varepsilon_{j'}x'=\sum_{s,t=1}^2\varepsilon_{t}l_{t;js}\theta_{sj'}(x)x',
\qquad\qquad\quad \forall x,x'\in E, j,j'=1,2,
\end{flalign*}
is an algebra, called the \emph{twisted direct product of $E\times E$ by $\theta$}, denoted by ${^{\theta_\times}(E\times E)}$.
\end{corollary}

In particular, if $\theta:E\to M_2(E)$ is the diagonal map, then $\theta_\times=(\theta,\varepsilon)$ is always a twisting system of $E\times E$ for any invertible $\kk$-basis $\varepsilon$ and ${^{\theta_\times} (E\times E)}$ is just the usual direct product $E\times E$.

\begin{remark}
To distinguish from Zhang-twisting system, we add a subscript ``$\times$'' for $\theta$ to represent a twisting system of  a direct product.
\end{remark}

Let $E$ be a $G$-graded algebra where $G$ is an abelian group. Suppose $\theta_\times=(\theta,\varepsilon)$ is a twisting system of $E\times E$, where  each  $\theta_{jj'}$ is a $G$-graded linear map for any $j,j'=1,2$.  Then the twisted  direct product  ${^{\theta_\times} (E\times E)}$ is a $G$-graded algebra, where for any $g\in G$,
\begin{align*}
	{^{\theta_\times}(E\times E)}_g=E_g\oplus E_g.
\end{align*}

\section{Double Ore extensions}

In this section, we recall the definition of double Ore extensions in \cite{ZZ1} and show some basic properties of noncommutative quadric hypersurfaces defined by double Ore extensions.

\begin{definition}
We say $B$ is a \emph{$\Z$-graded right  double Ore extension} of a $\Z$-graded algebra $A$ if $B$ generated by $A$ and two homogeneous variables $y_1,y_2$  satisfying
\begin{enumerate}
\item a homogeneous relation $y_2y_1=p_{12}y_1y_2+p_{11}y^2_1+\tau_1y_1+\tau_2y_2+\tau_0$ for some $p_{12},p_{11}\in\kk$ and homogeneous elements $\tau_0,\tau_1,\tau_2\in A$,

\item $B$ is a $\Z$-graded free left $A$-modules with basis $\{y_1^iy_2^j\mid i,j\geq0\}$,

\item $y_1A+y_2A\subseteq Ay_1+A y_2+A$.
\end{enumerate}
\end{definition}
Similarly, one can define a \emph{$\Z$-graded left  double Ore extension}. An algebra $B$ is called a \emph{$\Z$-graded double Ore extension} of $A$ if it is a $\Z$-graded left and a $\Z$-graded right double Ore extension of $A$ with the same generating set $\{y_1, y_2\}.$

The condition (3) of the definition of $\Z$-graded right double Ore extension is equivalent to the existence of an algebra homomorphism $\sigma=(\sigma_{ij}):A\to M_2(A)$ and a $\sigma$-derivation $\delta=(\delta_i):A\to A^{\oplus2}$ where all $\sigma_{ij}$ and $\delta_i$ are $\Z$-graded linear transformations of $A$ for $i,j=1,2$, such that 
$$
y_ia=\sigma_{i1}(a)y_1+\sigma_{i2}(a)y_2+\delta_i(a),
$$
for any $a\in A$ and $i=1,2$. We write  $A_{P}[y_1,y_2;\sigma,\delta,\tau]$ for a $\Z$-graded right double Ore extension where $P=\{p_{12},p_{11}\}$ and $\tau=\{\tau_0,\tau_1,\tau_2\}$, and write $A_{P}[y_1,y_2;\sigma]$ for convenience if $\tau=\{0,0,0\}$ and $\delta=0$.

We will focus on the trimmed case $A_{P}[y_1,y_2;\sigma]$ in the sequel. By \cite[Lemma 1.10]{ZZ1}, $\sigma$ in the trimmed case $A_{P}[y_1,y_2;\sigma]$ satisfies  following conditions:
\begin{equation}\label{eq: conditio for sigma}
\begin{aligned}
	&\sigma_{21}\sigma_{11}=p_{12}\sigma_{11}\sigma_{21},\\
	&\sigma_{22}\sigma_{12}=p_{12}\sigma_{12}\sigma_{22},\\
&\sigma_{21}\sigma_{12}+p_{12}\sigma_{22}\sigma_{11}=p_{12}\sigma_{11}\sigma_{22}+p^2_{12}\sigma_{12}\sigma_{21}.
\end{aligned}
\end{equation}

There is a conjecture about the relation between the invertibility of $\sigma$  and
the condition a right double Ore extension being a double Ore extension (see \cite{ZZ1}). It is true in the connected graded case by \cite[Lemma 1.9 and Proposition 1.13]{ZZ1}.

\begin{lemma}\label{lem: invertible iff double ore extension}
Let $A$ be a locally finite connected $\Z$-graded algebra and $B=A_P[y_1,y_2;\sigma]$ be a connected $\Z$-graded right double Ore extension of $A$ where $p_{12}\neq0$. Then $B$ is  a $\Z$-graded  double Ore extension of $A$ if and  only if $\sigma$ is invertible.
\end{lemma}

Let $A=T(V)/(R)$ be a noetherian Koszul Artin-Schelter regular algebra and $B=A_P[y_1,y_2;\sigma]$ be a $\Z$-graded double Ore extension of $A$ with $p_{12}\neq 0$ and $\deg y_1=\deg y_2=1$. By \cite[Theorem 2.1]{ZVZ} (or, \cite[Proposition 2.10]{SZL}) and \cite[Theorem 0.2]{ZZ1}, $B$ is a Koszul Artin-Schelter regular algebra. Clearly, $J=k\langle y_1,y_2\rangle/(y_2y_1-p_{12}y_1y_2-p_{11}y_1^2)$ is a $\Z$-graded subalgebra of $B$ and a noetherian Koszul Artin-Schelter regular algebra. We identify elements of $J$ as the elements of $B$ and write $R_J$ for the vector space spanned by $y_2y_1-p_{12}y_1y_2-p_{11}y_1^2$ in the sequel.  The following result is clear.

\begin{lemma}\label{lem: properties for B!} ~
\begin{enumerate}
	\item $
	B^{!}=T(V^*\oplus \kk y^*_1\oplus\kk y^*_2)/(R^\perp\oplus R^\perp_J \oplus R_\tau),
	$
	where
	\begin{align*}
	R_J^\perp&=\text{Span}_{\kk}\{(y_2^*)^2,y_1^*y_2^*+p_{12}y_2^*y_1^*,(y_1^*)^2+p_{11}y_2^*y_1^*\},\\
	R_\tau&=\text{Span}_{\kk}\{v^*y_i^*+y_1^*\sigma_{1i}^*(v^*)+y_2^*\sigma_{2i}^*(v^*)\mid v^*\in V^*, i=1,2\};
	\end{align*}
\item $A^!$ and $J^!$ are $\Z$-graded subalgebras of $B^{!}$;
\item $B^{!}$ is a $\Z$-graded free $A^!$-module and a $\Z$-graded free $J^!$-module on both sides;
\item $B^{!}$ has a $\kk$-basis $\{e,y_1^*e,y_2^*e,y_1^*y_2^*e\mid e\in \mathcal{B}_A\}$ where $\mathcal{B}_A$ is a $\kk$-basis of $A^!$.
\end{enumerate}
\end{lemma}

\begin{lemma}\label{lem: rep for Clifford deformation of double ore extensions}
Retain the notations above. Suppose $B$ is noetherian and $z+h$ is a nonzero regular central element of $B$ for some $z\in A_2$ and $h\in J_2$.  Write $C_{A^!}(z)=T(V^*)/(\widetilde{R^\perp})$ and $C_{J^!}(h)=T(\kk y^*_1\oplus\kk y^*_2)/(\widetilde{R^\perp_J})$, where
$$
\widetilde{R^\perp}=\{f-f(\hat{z})\mid f\in R^\perp \},\qquad
\widetilde{R^\perp_J}=\{g-g(\hat{h})\mid g\in R^\perp_J \},
$$
and $\hat{z}\in V^{\otimes 2}$ and $\hat{h}\in (\kk y_1\oplus\kk y_2)^{\otimes2}$ are lifts of $z$ and $h$ respectively. Then 

\begin{enumerate}	
\item$
	C_{B^!}(z+h)\cong T(V^*\oplus \kk y^*_1\oplus\kk y^*_2)/(\widetilde{R^\perp}\oplus \widetilde{R^\perp_J} \oplus R_\tau)
	$ as $\Z_2$-graded algebras;
\item $C_{A^!}(z)$ and $C_{J^!}(h)$ are $\Z_2$-graded subalgebras of $C_{B^!}(z+h)$;
\item $C_{B^!}(z+h)$ is a  $\Z_2$-graded  free $C_{A^!}(z)$-module and a  $\Z_2$-graded free $C_{J^!}(h)$-module on both sides;

\item $\gldim_{\Z_2} C_{B^!}(z+h)\geq \max\{\gldim_{\Z_2} C_{A^!}(z),\gldim_{\Z_2} C_{J^!}(h)\}$.
\end{enumerate}
\end{lemma}
\begin{proof}
Clearly, $z\in A_2$ and $h\in J_2$ are both central elements of $A$ and $J$ respectively.

(1,2) follow by the definition of Clifford deformations and Lemma \ref{lem: properties for B!}.

(3) Since $C_{A^!}(z)$, $C_{J^!}(h)$ and  $C_{B^!}(z+h)$ are Clifford deformations, a class of PBW deformations,  they have the same forms of $\kk$-bases of $A^!$, $J^!$ and $B^!$ respectively. By Lemma \ref{lem: properties for B!}(4), the map of multiplication
$$
C_{J^!}(h)\otimes C_{A^!}(z)\to C_{B^!}(z+h)
$$
is an isomorphism as left $\Z_2$-graded $C_{J^!}(h)$-modules and right $\Z_2$-graded $C_{A^!}(z)$-modules. 

Similarly, one can check that the natural map $
C_{A^!}(z)\otimes C_{J^!}(h)\to C_{B^!}(z+h)
$
is an isomorphism as right $\Z_2$-graded $C_{J^!}(h)$-modules and left $\Z_2$-graded $C_{A^!}(z)$-modules.

(4) It follows by a graded version of \cite[Theorem 7.2.6]{MR}.
\end{proof}

Our main goal is to present a skew version of the classical Kn\"orrer's periodicity theorem. To this end, we consider noncommutative quadric hypersurface having the form $B/(z+y_1^2+y_2^2)$ in the rest of this paper where $z\in A_2$ is a nonzero regular central element. To make such class of noncommutative quadric hypersurfaces well defined, $y_1^2+y_2^2$ should be a central element of $B$ and also of $J$. By an easy calculation, we have the following result.

\begin{lemma}\label{lem: p12=+-1}
Suppose the element $y^2_1+y^2_2$ is central in $B$, then 
	$P=\{1,0\}$ or $P=\{-1,p_{11}\}$.
\end{lemma}

With respect to the values of $p_{12}$, we discuss in two cases in the following.

\section{A Skew version of Kn\"orrer's Periodicity Theorem with $p_{12}=1$}
Let $B=A_{\{1,0\}}[y_1,y_2;\sigma]$ be a $\Z$-graded double Ore extension of a noetherian Koszul Artin-Schelter regular algebra $A=T(V)/(R)$  with $\deg y_1=\deg y_2=1$, and  $z\in A_2$ be a nonzero regular central element of $A$. It is not hard to obtain the following conditions for $\sigma$.

\begin{lemma}\label{lem: +1case condition for central element}
 The element  $z+y_1^2+y_2^2$ is a central element of $B$ if and only if 
\begin{enumerate}
	\item $\sigma_{11}^2+\sigma_{21}^2=\sigma_{12}^2+\sigma_{22}^2=\id_A$;
	\item 
	$\sigma_{11}\sigma_{12}+\sigma_{21}\sigma_{22}+\sigma_{12}\sigma_{11}+\sigma_{22}\sigma_{21}=0$;
	\item $\sigma(z)=
	\begin{pmatrix}
		z &	0\\
		0	&	z
	\end{pmatrix}.
	$
\end{enumerate}	
\end{lemma}

By Lemma \ref{lem: +1case condition for central element}(3) and the construction \eqref{eq: dual map of simga from clifford deformation to its matrix}, we have the following algebra homomorphism
$$
\sigma^!=(\sigma_{ij}^!):C_{A^!}(z)\to M_2(C_{A^!}(z)),
$$
where any $\sigma_{ij}^!$ is a $\Z_2$-graded map for $i,j=1,2$. Then
\begin{corollary}\label{cor: properties for dual of sigma}
\begin{enumerate}
	\item $(\sigma^!_{11})^2+(\sigma^!_{21})^2=(\sigma^!_{12})^2+(\sigma^!_{22})^2=\id_{C_{A^!}(z)};$
	\item 
	$\sigma^!_{12}\sigma^!_{11}+\sigma^!_{22}\sigma^!_{21}+\sigma^!_{11}\sigma^!_{12}+\sigma^!_{21}\sigma^!_{22}=0;$
	\item $\sigma^!_{11}\sigma^!_{21}=\sigma^!_{21}\sigma^!_{11};$
	\item $\sigma^!_{12}\sigma^!_{22}=\sigma^!_{22}\sigma^!_{12};$
	\item $\sigma^!_{22}\sigma^!_{11}-\sigma^!_{12}\sigma^!_{21}=\sigma^!_{11}\sigma^!_{22}-\sigma^!_{21}\sigma^!_{12};$
 \item   $\sigma^!$ is $t$-invertible. 
\end{enumerate}	
\end{corollary}
\begin{proof}
(1)--(5) follow by Lemma \ref{lem: +1case condition for central element} and \eqref{eq: conditio for sigma}.  (6) holds by Lemma \ref{lem: invertible iff double ore extension} and Lemma \ref{lem: inv of sigma and simga dual},
\end{proof}

For now on, we construct a twisting system on $M_2(C_{A^!}(z))$. We choose an invertible $\Z_2$-graded basis $\II=\{\II^{(0)}_1,\II^{(0)}_2,\II^{(1)}_1,\II^{(1)}_2\}$  of $M_2(\kk)$ as follows,
\begin{equation}\label{eq: definition of basis}
\begin{array}{cccc}
	\II^{(0)}_{1}=\begin{pmatrix}
	1 & 0\\
	0 & 1
\end{pmatrix},&
\II^{(0)}_{2}=\begin{pmatrix}
	-\sqrt{-1} & 0\\
	0 & \sqrt{-1}
\end{pmatrix},&
\II^{(1)}_{1}=\begin{pmatrix}
	0 & 1\\
	1 & 0
\end{pmatrix},
&\II^{(1)}_{2}=\begin{pmatrix}
	0 & \sqrt{-1}\\
	-\sqrt{-1} & 0
\end{pmatrix}.
\end{array}
\end{equation}
Define linear maps $\theta^{(0)}$ and $\theta^{(1)}$ from $C_{A^!}(z)$ to $M_2(C_{A^!}(z))$ by 
\begin{equation}\label{eq: definition of theta}
\theta^{(0)}=(\theta^{(0)}_{jj'})=\begin{pmatrix}
	\id & \sigma^!_{12}\sigma^!_{11}+\sigma^!_{22}\sigma^!_{21}\\
	0& \sigma^!_{22}\sigma^!_{11}-\sigma^!_{12}\sigma^!_{21}
\end{pmatrix},\quad
\theta^{(1)}=(\theta^{(1)}_{jj'})=\sigma^!\circ\xi_{-1}=(\sigma^!_{ij}\circ\xi_{-1}).
\end{equation}
Clearly,  $\theta^{(i)}_{jj'}$ is a $\Z_2$-graded linear map for any $i\in\Z_2,j,j'=1,2$. Those two linear maps satisfy  the following properties.

\begin{lemma}\label{lem: properties of pair of theta} Retain the notations above. Then
\begin{enumerate}	\item  $\theta^{(0)}_{12}(aa')=a\theta^{(0)}_{12}(a')+\theta^{(0)}_{12}(a)\theta^{(0)}_{22}(a')$ for any $a,a'\in C_{A^!}(z)$;	
\item $\theta^{(0)}$ and $\theta^{(1)}$ are algebra homomorphisms;
\item $\theta^{(0)}_{12}\theta^{(0)}_{22}+\theta^{(0)}_{22}\theta^{(0)}_{12}=0$;
	
\item $(\theta^{(0)}_{22})^2-(\theta^{(0)}_{12})^2=\id_{C_{A^!}(z)}$;
\end{enumerate}
\end{lemma}
\begin{proof}
(1) holds since $\sigma^!$ is an algebra homomorphism. (2) follows by (1). (3,4) follow by Corollary \ref{cor: properties for dual of sigma}.
\end{proof}

\begin{theorem}\label{thm: +1case, clifford def is twisted matrix algebra}
Let $\theta^{(0)}$ and $\theta^{(1)}$ be linear maps from $C_{A^!}(z)$ to $M_2(C_{A^!}(z))$ defined in \eqref{eq: definition of theta} and $\II$ be the basis of $M_2(\kk)$ defined in \eqref{eq: definition of basis}. Then
\begin{enumerate}
	\item  $\Theta=(\{\theta^{(i)}\}_{i\in\Z_2},\II)$ is a twisting system of $M_2(C_{A^!}(z))$;
	\item $C_{B^!}(z+y_1^2+y_2^2)\cong {^\Theta M_2(C_{A^!}(z))}$ as $\Z_2$-graded algebras where the $\Z_2$-grading of  
 ${^\Theta M_2(C_{A^!}(z))}$ comes from the total degrees (see Remark \ref{rem: Z_2-grading of tiwsted matrix algebra}(1)).
\end{enumerate}
\end{theorem}
\begin{proof}
(1) Write $\varphi$ for the inverse of $\sigma$. By Lemma \ref{lem: inv of sigma and simga dual} and its proof, $\varphi^!=(\varphi^!_{ij})$ exits and is the $t$-inverse of $\sigma^!$. Then one obtains that  $\theta^{(0)}$ and $\theta^{(1)}$ are $t$-invertible with the $t$-inverses
$$
\begin{pmatrix}
	\id  & 0 \\\varphi_{21}^!\varphi_{11}^!+\varphi_{22}^!\varphi_{12}^!&\varphi_{22}^!\varphi_{11}^!-\varphi_{21}^!\varphi_{12}^!
\end{pmatrix},\qquad
\varphi^!\circ \xi_{-1},
$$
respectively. Clearly, 
$\theta^{(1)}(1)=I_2\in GL_2(\kk)$, and the invertible matrices defined in \eqref{eq: symbol of L} are $L^{(ii')}_1=I_2$ for any $i,i'\in \Z_2$, and 
$$
L^{(00)}_2=\begin{pmatrix}
	0 & -1\\
	1 & 0
\end{pmatrix},\ 
L^{(01)}_2=\begin{pmatrix}
	0 & 1\\
	-1 & 0
\end{pmatrix},\ 
L^{(10)}_2=\begin{pmatrix}
	0 & -1\\
	1 & 0
\end{pmatrix},\ 
L^{(11)}_2=\begin{pmatrix}
	0 & 1\\
	-1 & 0
\end{pmatrix}.
$$
Then by Corollary \ref{cor: properties for dual of sigma} and Lemma \ref{lem: properties of pair of theta}, it is straightforward but tedious to check $\Theta$ satisfies the following simplified equations of Proposition \ref{prop: equivalent condition for theta}(3) 
\begin{align*}
	\theta^{(i')}_{pj'}\left(\theta^{(i)}_{11}(x)x'\right)+
	l^{(ii')}_{p2u }\theta^{(i')}_{uj'}\left(\theta^{(i)}_{21}(x)x'\right)=&\theta^{(i+i')}_{p1}(x)\theta^{(i')}_{1j'}(x')+
	\theta^{(i+i')}_{p2}(x)\theta^{(i')}_{2j'}(x'),\\
	\theta^{(i')}_{pj'}\left(\theta^{(i)}_{12}(x)x'\right)+
	l^{(ii')}_{p2u }\theta^{(i')}_{uj'}\left(\theta^{(i)}_{22}(x)x'\right)=&l^{(ii')}_{122}\theta^{(i+i')}_{p1}(x)\theta^{(i')}_{2j'}(x')+
	l^{(ii')}_{221}\theta^{(i+i')}_{p2}(x)\theta^{(i')}_{1j'}(x'),
\end{align*}
where $u=\{1,2\}-\{p\}$ for any $x,x'\in E$, $i,i'\in \Z_2$ and $p=1,2$.

(b)	By Lemma \ref{lem: rep for Clifford deformation of double ore extensions}(1), $C_{B^!}(z+y_1^2+y_2^2)$ is generated by $V^*,y^*_1,y^*_2$ subject to the relations consisting of the relations of $C_{A^!}(z)$, and 
\begin{align*}
&(y_1^*)^2-1,(y_2^*)^2-1,y_1^*y_2^*+y_2^*y_1^*,\\ &v^*y_i^*+y_1^*\sigma_{1i}^*(v^*)+y_2^*\sigma_{2i}^*(v^*), \qquad \forall v^*\in V^*,i=1,2.
\end{align*}
We list the precise multiplication of  ${^{\Theta}M_2(C_{A^!}(z))}$ below: for any $x,x'\in C_{A^!}(z)$,
\begin{table}[h]
    \centering
        \renewcommand\arraystretch{2}
    \begin{tabular}{c|cccc}
       ~ & $\II^{(0)}_1x'$& $\II^{(0)}_2x' $&$\II^{(1)}_1x'$ & $\II^{(1)}_2x'$ \\ \hline
       $ \II^{(0)}_1x$ & $\II^{(0)}_1xx'$ & \makecell{$\II^{(0)}_1\theta^{(0)}_{12}(x)x'$\\$+\II^{(0)}_2\theta^{(0)}_{22}(x)x'$}& \makecell{$\II^{(1)}_1\theta^{(1)}_{11}(x)x'$\\$+\II^{(1)}_2\theta^{(1)}_{21}(x)x'$} & \makecell{$\II^{(1)}_1\theta^{(1)}_{12}(x)x'$\\$+\II^{(1)}_2\theta^{(1)}_{22}(x)x' $}\\ 
        $\II^{(0)}_2x$ &$ \II^{(0)}_2xx' $& \makecell{$-\II^{(0)}_1\theta^{(0)}_{22}(x)x'$\\$+\II^{(0)}_2\theta^{(0)}_{12}(x)x'$} & 
        \makecell{$\II^{(1)}_1\theta^{(1)}_{21}(x)x'$\\$-\II^{(1)}_2\theta^{(1)}_{11}(x)x' $}&\makecell{ $\II^{(1)}_1\theta^{(1)}_{22}(x)x'$\\$-\II^{(1)}_2\theta^{(1)}_{12}(x)x' $}\\ 
      
        $\II^{(1)}_1x $&$ \II^{(1)}_1xx'$ &\makecell{$ \II^{(1)}_1\theta^{(0)}_{12}(x)x'$\\$+\II^{(1)}_2\theta^{(0)}_{22}(x)x'$ }& \makecell{$\II^{(0)}_1\theta^{(1)}_{11}(x)x'$\\$+\II^{(0)}_2\theta^{(1)}_{21}(x)x'$} &\makecell{$ \II^{(0)}_1\theta^{(1)}_{12}(x)x'$\\$+\II^{(0)}_2\theta^{(1)}_{22}(x)x'$} \\
    $\II^{(1)}_2x $& $\II^{(1)}_2xx'$ & \makecell{$-\II^{(1)}_1\theta^{(0)}_{22}(x)x'$\\$+\II^{(1)}_2\theta^{(0)}_{12}(x)x'$} &\makecell{ $\II^{(0)}_1\theta^{(1)}_{21}(x)x'$\\$-\II^{(0)}_2\theta^{(1)}_{11}(x)x' $}&\makecell{$ \II^{(0)}_1\theta^{(1)}_{22}(x)x'$\\$-\II^{(0)}_2\theta^{(1)}_{12}(x)x'$}     \\   
    \end{tabular}
         \setlength{\belowcaptionskip}{0cm}
         \vspace{0.1cm}
       \caption{Multiplication Table of ${^\Theta M_2(C_{A^!}(z))}$.}\label{tab: mul of twisted matrix of C_A}
\end{table}

Then one obtains that there exists a 
$\Z_2$-graded algebra isomorphism from $C_{B^!}(z+y_1^2+y_2^2)$ to ${^\Theta M_2(C_{A^!}(z))}$ mapping
$$y^*_1\mapsto \II^{(1)}_1,\qquad y^*_2\mapsto \II^{(1)}_2,\qquad v^*\mapsto \II^{(0)}_1v^*,\ \forall v^*\in V^*.$$
\end{proof}

\begin{proposition}\label{prop: e is a full idempotent} Retain the assumptions of Theorem \ref{thm: +1case, clifford def is twisted matrix algebra}. The homogeneous element $e=\begin{pmatrix}1&0\\0&0\end{pmatrix}$ is a full idempotent of 
${^\Theta M_2(C_{A^!}(z))}$.

As a consequence,  $\gr_{\Z_2} {^\Theta M_2(C_{A^!}(z))}$ and  $\gr_{\Z_2} e\left({^\Theta M_2(C_{A^!}(z))}\right)e$ is equivalent.
\end{proposition}
\begin{proof}
 Note that $	e=2^{-1}\left(\II^{(0)}_1+\sqrt{-1}\II^{(0)}_2\right)$. By Table \ref{tab: mul of twisted matrix of C_A}, one obtains 
$$
2^{-1}\left(\II^{(i)}_1+\sqrt{-1}\II^{(i)}_2\right)\star e =	2^{-1}\left(\II^{(i)}_1+\sqrt{-1}\II^{(i)}_2\right),
$$
and
\begin{align*}
\II^{(i)}_1x&=2^{-1}\left(\II^{(i)}_1+\sqrt{-1}\II^{(i)}_2\right)\star e \star \II^{(0)}_1x+ 2^{-1}\left(\II^{(i+1)}_1+\sqrt{-1}\II^{(i+1)}_2\right)\star e\star \II^{(1)}_1x,\\
\II^{(i)}_2x&=
2^{-1}\left(\II^{(i+1)}_1+\sqrt{-1}\II^{(i+1)}_2\right)\star e\star \sqrt{-1}\II^{(1)}_1x-2^{-1}\left(\II^{(i)}_1+\sqrt{-1}\II^{(i)}_2\right)\star e \star \sqrt{-1}\II^{(0)}_1x,
\end{align*}
for any $i\in \Z_2$ and $x\in C_{A^!}(z)$. Hence,  $e$ is a full idempotent.
\end{proof}

To understand the multiplication of ${^\Theta M_2(C_{A^!}(z))}$ further more, we introduce four linear transformations of $C_{A^!}(z)$: 
$$
\begin{array}{ll}
\Xi_1=2^{-1}(\id+\sqrt{-1}\theta^{(0)}_{12}+\theta^{(0)}_{22}), & \Xi_2=2^{-1}(\id+\sqrt{-1}\theta^{(0)}_{12}-\theta^{(0)}_{22}),\\
\Phi_{1}=2^{-1}(\theta^{(1)}_{11}-\sqrt{-1}\theta^{(1)}_{12}+\sqrt{-1}\theta^{(1)}_{21}+\theta^{(1)}_{22}),&
	\Phi_{2}=2^{-1}(\theta^{(1)}_{11}-\sqrt{-1}\theta^{(1)}_{12}-\sqrt{-1}\theta^{(1)}_{21}-\theta^{(1)}_{22}).
\end{array}
$$
We show some parts of the multiplication of ${^\Theta M_2(C_{A^!}(z))}$ in the usual matrix forms involving the linear transforms listed above: 
\begin{equation}\label{eq: mul of twisted matrix}
\begin{array}{ll}
	\begin{pmatrix}
		a  &  0\\
		0  &   0
	\end{pmatrix}*
	\begin{pmatrix}
		a'  &  0\\
		0  &   0
	\end{pmatrix}=
	\begin{pmatrix}
		\Xi_1(a)a'  &  0\\
		0  &   0
	\end{pmatrix},
	&
	\begin{pmatrix}
		a  &  0\\
		0  &   0
	\end{pmatrix}*
	\begin{pmatrix}
		0  &  a'\\
		0  &   0
	\end{pmatrix}=
	\begin{pmatrix}
		0 &  \Phi_1(a)a'\\
		0  &   0
	\end{pmatrix},
	\\\\
	\begin{pmatrix}
		0  &  a\\
		0  &   0
	\end{pmatrix}*
	\begin{pmatrix}
		a'  &  0\\
		0  &   0
	\end{pmatrix}=
	\begin{pmatrix}
		0 &  \Xi_2(a)a'\\
		0  &   0
	\end{pmatrix},
	&
	\begin{pmatrix}
		0  &  a\\
		0  &   0
	\end{pmatrix}*
	\begin{pmatrix}
		0  &  a'\\
		0  &   0
	\end{pmatrix}=
	\begin{pmatrix}
		\Phi_2(a)a' &  0\\
		0  &   0
	\end{pmatrix}.
\end{array}
\end{equation}
Here are some useful properties of those four liner transformations.
\begin{lemma}\label{lem: properties of Xi Phi}
	\begin{enumerate}
		\item $\Xi_1^2=\Xi_1, \Xi_2^2=\Xi_2$. 
		\item $\Xi_1(ab)=a\Xi_2(b)+\Xi_1(a)\theta^{(0)}_{22}(b)=a\Xi_1(b)+\Xi_1(a)\theta^{(0)}_{22}(b)-a\theta^{(0)}_{22}(b)$ for any $a,b\in C_{A^!}(z)$.
		\item $\Xi_2(ab)=a\Xi_2(b)+\Xi_2(a)\theta^{(0)}_{22}(b)=a\Xi_1(b)+\Xi_2(a)\theta^{(0)}_{22}(b)-a\theta^{(0)}_{22}(b)$  for any $a,b\in C_{A^!}(z)$.

		\item $\Xi_1\Phi_{1}=\theta^{(0)}_{22} \Phi_{1}$, 
 $\Phi_{1}\Xi_1=\Phi_{1},$ 
  $\Phi_{1}\Xi_2=\Phi_{1}(\sqrt{-1}\theta^{(0)}_{12})$,
  $\Xi_2\Phi_{1}=0.$

		\item $\Xi_1\Phi_{2}=0$, $\Phi_{2}\Xi_1=\Phi_{2}(\sqrt{-1}\theta^{(0)}_{12})$, $\Phi_{2}\Xi_{2}=\Phi_{2}$,
 $\Xi_2\Phi_{2}=-\theta^{(0)}_{22}\Phi_{2}$.
		 \item $\Phi_{1}(\Xi_1(a)b)=\Phi_{1}(a)\Phi_{1}(b)$, $\Phi_{1}(a\Xi_1(b))=\Phi_{1}(a)\Phi_{1}((\Xi_1(b))$, $\Phi_{1}(\Phi_{2}(a)\Xi_2(b))=\Xi_2(a)\Phi_{2}(b)$ for any $a,b\in C_{A^!}(z)$.

		\item $\Phi_{2}(\Xi_2(a)b)=\Phi_{2}(a)\Phi_{1}(b)$, 		
		$\Phi_{2}(\Phi_{1}(a)\Xi_2(b))=\Xi_{1}(a)\Phi_{2}(b)$ for any $a,b\in C_{A^!}(z)$. 				
				\item
				$(\theta^{(1)}_{11}-\sqrt{-1}\theta^{(1)}_{21})\Phi_{1}=(\theta^{(1)}_{22}+\sqrt{-1}\theta^{(1)}_{12})\Phi_{1}=\Xi_1$.
    \item $(\theta^{(1)}_{11}+\sqrt{-1}\theta^{(1)}_{21})\Phi_{2}=(\sqrt{-1}\theta^{(1)}_{12}-\theta^{(1)}_{22})\Phi_{2}=\Xi_2$.
	\end{enumerate}
\end{lemma}
\begin{proof}
All follow by Corollary \ref{cor: properties for dual of sigma} and Lemma \ref{lem: properties of pair of theta}.
 \end{proof}

By Lemma \ref{lem: properties of Xi Phi}(1), we write $S$ for the $\mathbb{Z}_2$-graded eigenspace for the eigenvalue $1$ of $\Xi_1$ and $M$ for the $\mathbb{Z}_2$-graded eigenspace for the eigenvalue $1$ of $\Xi_2$.

\begin{lemma}
	\begin{enumerate}
		\item The $\mathbb{Z}_2$-graded eigenspace $S$ is a $\mathbb{Z}_2$-graded subalgebra of $C_{A^!}(z)$ and $S\cong {\Phi_{1}}(S)$ as $\mathbb{Z}_2$-graded algebras.

		\item   The $\mathbb{Z}_2$-graded eigenspace $M$ is a $\mathbb{Z}_2$-graded $S$-bimodule with the $S$-action defined by
  $$
  s\cdot m\cdot s'=\Phi_{1}(s)ms',
  $$
  for any $s,s'\in S$ and $m\in M$.

		\item For any $m,m'\in M$, $\Phi_2(m)m'\in S$. 
		
		\item The map  $$
		\begin{array}{cclc}
			\psi:& M(1)\otimes_S M(1)&\to & S\\
			~&m\otimes m'&\mapsto & \Phi_{2}(m)m',
		\end{array}
		$$
		is a $\Z_2$-graded $S$-bimodule homomorphism satisfying $m\psi(m'\otimes m'')=\psi(m\otimes m')\cdot m''$ for any homogeneous elements  $m,m',m''\in M$.

	\end{enumerate}
\end{lemma}
\begin{proof}
(1) Clearly $\Xi_1(1)=1\in S$. By Lemma \ref{lem: properties of Xi Phi}(2), for any $s,s'\in S$,
 \begin{align*}
 	\Xi_1(ss')=s\Xi_1(s')+\Xi_1(s)\theta^{(0)}_{22}(s')-s\theta^{(0)}_{22}(s')=ss'.
 \end{align*}
Hence, $S$ is a $\Z_2$-graded subalgebra of $C_{A^!}(z)$.

By Lemma \ref{lem: properties of Xi Phi}(6), the restriction map ${\Phi_{1}}_{\mid S}:S\to \Phi_{1}(S)$ is a surjective $\Z_2$-graded algebra homomorphism. If $s\in \Ker{\Phi_{1}}_{\mid S}$, then $s=\Xi_1(s)=(\theta^{(1)}_{11}-\sqrt{-1}\theta^{(1)}_{21})\Phi_{1}(s)=0$ by Lemma \ref{lem: properties of Xi Phi}(8). So ${\Phi_{1}}_{\mid S}$ is an isomorphism.

(2) By Lemma \ref{lem: properties of Xi Phi}(3,4),
\begin{align*}
\Xi_2(\Phi_{1}(s) ms')=&\Phi_{1}(s)\Xi_2(ms')+\Xi_2(\Phi_{1}(s))\theta^{(0)}_{22}(ms')\\
	=&\Phi_{1}(s)\left(m\Xi_2(s')+\Xi_2(m)\theta^{(0)}_{22}(s')-m\theta^{(0)}_{22}(s')\right)\\	
	=&\Phi_{1}(s)ms',
\end{align*}
 for any $s,s'\in S$ and $m\in M$. It implies that $M$ is a $(\Phi_1(S),S)$-bimodule. By (1), the result follows.
 
(3) By Lemma \ref{lem: properties of Xi Phi}(2,5), $
 	\Xi_1(\Phi_{2}(m)m')=\Phi_{2}(m)\Xi_2(m')+\Xi_1(\Phi_{2}(m))\theta^{(0)}_{22}(m')=\Phi_{2}(m)m'
$,  for any $m,m'\in M$.

(4) For any $s\in S$ and $m,m'\in M(1)$, 
$$\psi(ms\otimes m')=\Phi_2(\Xi_2(m)s)m'=\Phi_2(m)\Phi_1(s)m'=\Phi_2(m)(s\cdot m')=\psi(m\otimes s\cdot m'),
$$
by Lemma \ref{lem: properties of Xi Phi}(7). Combining with (3), $\psi$ is well defined.

For any $s,s'\in S$ and $m,m'\in M(1)$, 
$$
	\psi(s\cdot m\otimes m'\cdot s')=\Phi_{2}(\Phi_1(s) \Xi_2(m))m's'
	=\Xi_{1}(s) \Phi_2(m)m's'=s\psi(m\otimes m')s',
$$
by Lemma \ref{lem: properties of Xi Phi}(7).
\end{proof}

The last lemma allows us to construct a semi-trivial extension 
$$\Lambda=S\ltimes_{\psi}{M(1)},$$
which is a $\Z_2\times \Z_2$-graded algebra and also a $\Z_2$-graded algebra by forgetting the first $\Z_2$-grading (see  Remark \ref{rem: semi-trivial extension} (3)). In general, $\Lambda$ is not strongly $\Z_2$-graded (see Example \ref{ex: +1case two setp Ore extensions} below).

\begin{theorem}\label{thm: +1case equivalence to Lambda}
	Let $B=A_{\{1,0\}}[y_1,y_2;\sigma]$ be a $\Z$-graded double Ore extension of a noetherian Koszul Artin-Schelter regular algebra $A$ with $\deg y_1=\deg y_2=1$ and $z\in A_2$ be a nonzero regular central element of $A$. Suppose $B$ is noetherian and  $z+y_1^2+y_2^2$ is a central element of $B$. Then there are equivalences of triangulated categories
	$$\umcm(B/(z+y_1^2+y_2^2))\cong D^b(\gr_{\Z_{2}}{^\Theta M_2(C_{A^!}(z))})\cong D^b(\module {^\Theta M_2(C_{A^!}(z))}_0)\cong D^b(\gr_{\Z_2} \Lambda).$$
\end{theorem}
\begin{proof}
By Theorem \ref{thm: +1case, clifford def is twisted matrix algebra} and Theorem \ref{thm: properties of clifford def}, one obtains that there are equivalences of triangulated categories
$$\umcm(B/(z+y_1^2+y_2^2))\cong D^b(\gr_{\Z_{2}}{^\Theta M_2(C_{A^!}(z))})\cong D^b(\module {^\Theta M_2(C_{A^!}(z))}_0).$$

By Proposition \ref{prop: e is a full idempotent}, $\gr_{\Z_{2}}{^\Theta M_2(C_{A^!}(z))}$ is equivalent to $\gr_{\Z_2} e\left({^\Theta M_2(C_{A^!}(z))}\right)e$. By Table \ref{tab: mul of twisted matrix of C_A}, one obtains
\begin{align*}
    e\star\II^{(0)}_1x\star e&=2^{-1}\left(
	\II^{(0)}_1\Xi_1(x)+\sqrt{-1}\II^{(0)}_2\Xi_1(x)
	\right),\\
	e\star\II^{(0)}_2x\star e&=-2^{-1}\sqrt{-1}\left(\II^{(0)}_1\Xi_1(x)+\sqrt{-1}\II^{(0)}_2\Xi_1(x)\right),\\
	e\star\II^{(1)}_1x\star e&=2^{-1}\left(
	\II^{(1)}_1\Xi_2(x)-\sqrt{-1}\II^{(0)}_2\Xi_2(x)\right),\\
	e\star\II^{(1)}_2x\star e&=2^{-1}\sqrt{-1}\left(
	\II^{(1)}_1\Xi_2(x)-\sqrt{-1}\II^{(0)}_2\Xi_2(x)\right),
\end{align*}
for any $x\in C_{A^!}(z)$. It implies  
$$
e\left({^\Theta M_2(C_{A^!}(z))}\right)e
=\left\{\begin{pmatrix}
	\Xi_1(x)& \Xi_2(x')\\
	0 &0
\end{pmatrix}\mid x,x'\in C_{A^!}(z)
\right\}=\begin{pmatrix}
	S& M\\
	0 &0
\end{pmatrix}. 
$$
Note that $e\left({^\Theta M_2(C_{A^!}(z))}\right)_0e=\begin{pmatrix}
	S_0& M_1\\
	0 &0
\end{pmatrix}$ and $e\left({^\Theta M_2(C_{A^!}(z))}\right)_1e=\begin{pmatrix}
S_1& M_0\\
0 &0
\end{pmatrix}$. By \eqref{eq: mul of twisted matrix}, we have $e\left({^\Theta M_2(C_{A^!}(z))}\right)e\cong \Lambda$ as  $\Z_2$-graded algebras. 

Hence, there is an equivalence of triangulated categories
\begin{equation*}
D^b(\gr_{\Z_{2}}{^\Theta M_2(C_{A^!}(z))})\cong D^b(\gr_{\Z_2} \Lambda).\qedhere
\end{equation*}
\end{proof}

\begin{example}\label{ex: +1case two setp Ore extensions}
Let $A$ be a noetherian Koszul Artin-Schelter regular algebra. Let $\sigma_1$ and $\sigma_2$ be graded automorphisms of $A$ satisfying
$$
\sigma_1^2=\sigma_2^2=\id_A,\qquad\sigma_{1}\sigma_2=\sigma_2\sigma_1.
$$
Naturally, $\sigma_2$ can be extended to  be a graded automorphism of $A[y_1;\sigma_1]$ which preserves $y_1$, still denoted by $\sigma_2$. Then we have an Ore extension $A[y_1;\sigma_1][y_2;\sigma_2]$ is isomorphic to the double Ore extension  $B=A_{\{1,0\}}[y_1,y_2;\sigma]$, where
$$
\sigma=
\begin{pmatrix}
	\sigma_1 & 0\\
	0 &\sigma_2
\end{pmatrix}
:A\to M_2(A).
$$

Let $z\in A_2$ be a nonzero regular central element of $A$ such that $\sigma_1(z)=\sigma_2(z)=z$. By Lemma \ref{lem: +1case condition for central element}, $z+y_1^2+y_2^2$ is a central element of $B$. In this case, the linear maps in the twisting system are 
\begin{equation*}
	\theta^{(0)}=\begin{pmatrix}
		\id & 0\\
		0& \sigma^!_{22}\sigma^!_{11}
	\end{pmatrix},\quad
	\theta^{(1)}=\begin{pmatrix}
		\sigma_1^! & 0\\
		0& \sigma^!_{2}
	\end{pmatrix}\circ\xi_{-1}
,
\end{equation*}
and the liner transformations of $C_{A^!}(z)$ are
$$
\begin{array}{llll}
	\Xi_1=2^{-1}(\id+\sigma^!_{2}\sigma^!_{1}), & \Xi_2=2^{-1}(\id-\sigma^!_{2}\sigma^!_{1}),&
	\Phi_{1}=2^{-1}(\sigma^!_{1}+\sigma^!_{2})\xi_{-1},&
	\Phi_{2}=2^{-1}(\sigma^!_{1}-\sigma^!_{2})\xi_{-1}.
\end{array}
$$
So 
$$
S=\left\{x\in C_{A^!}(z)\mid \sigma^!_{1}(x)=\sigma^!_{2}(x)\right\},\qquad
M=\left\{x\in C_{A^!}(z)\mid \sigma^!_{1}(x)=-\sigma^!_{2}(x)\right\}, 
$$
and $
\Lambda=S\ltimes_{\psi}{M(1)}
$ with multiplication
$$
(s,m)\odot(s',m')=(ss'+\sigma^!_{1}\xi_{-1}(m)m',\sigma^!_{1}\xi_{-1}(s)m'+ms'),
$$
for any homogeneous elements $s,s'\in S$ and $m,m'\in M$. By Theorem \ref{thm: +1case equivalence to Lambda}, there is an equivalence of triangulated categories
	$$\umcm(B/(z+y_1^2+y_2^2)) \cong D^b(\gr_{\Z_2} \Lambda).$$
There are some special cases.
\begin{enumerate}
	\item Suppose $\sigma_1=\sigma_2=\sigma$. Then $
	S=C_{A^!}(z),
	M=0 
	$, and $\Lambda=C_{A^!}(z)$. Hence, there are equivalences of triangulated categories 
	$$\umcm(A[y_1;\sigma][y_2;\sigma]/(z+y_1^2+y_2^2)) \cong D^b(\gr_{\Z_2} C_{A^!}(z))\cong \umcm(A/(z)).$$
	
	It is exactly the noncommutative Kn\"orrer's periodicity theorem for noncommutative quadric hypersurfaces (see \cite[Theorem 1.3]{MU2}, or \cite[Theorem 0.4]{HY}).
	\item Suppose $\sigma_1=\xi_{-1}$ and $\sigma_2=\id$. Then $
	S_0=C_{A^!}(z)_0,S_1=0, M_0=0,M_1=C_{A^!}(z)_1
	$. So $\Z_2$-graded algebra $\Lambda$ is concentrated in degree $0$ with $\Lambda_0= C_{A^!}(z)$. Hence, there is an equivalence of triangulated categories  
	$$
	\umcm(A[y_1;\xi_{-1}][y_2]/(z+y_1^2+y_2^2)) \cong D^b(\module C_{A^!}(z))^{\times 2}.
	$$
	It is equivalent to \cite[Theorem 1.4]{MU2} (up to Zhang-twist) which provides a tool called two point reduction for computing a classes of noncommutative quadric hypersurfaces defined by $(\pm 1)$-skew polynomial algebras (see \cite[Lemma 6.18]{MU2}). We will also give another approach to \cite[Theorem 1.4]{MU2} in Example \ref{ex: -1case}.
\end{enumerate}
\end{example}

\begin{example}
Let $A=k_{-1}[x_1,x_2]$ and the nonzero regular central element $z=x_1^2+x_2^2$. Let $B=A_{\{1,0\}}[y_1,y_2;\sigma]$ be a $\Z$-graded double Ore extension of $A$ where $\sigma=(\sigma_{ij}):A\to M_2(A)$ satisfies
$$
\begin{array}{llll}
    \sigma_{11}(x_1)=hx_1, &  \sigma_{12}(x_1)=hx_2, & \sigma_{21}(x_1)=hx_2,& \sigma_{22}(x_1)=-hx_1, \\
        \sigma_{11}(x_2)=hx_2, &  \sigma_{12}(x_2)=hx_1, &  
    \sigma_{21}(x_2)=hx_1,& \sigma_{22}(x_1)=-hx_2,
\end{array}
$$
where $2h^2=1$. It is a special case of the class $\mathbb{Z}$ with $f=1$ in \cite{ZZ2}. By \cite[Proposition 5.16]{ZZ2}, $B$ is noetherian. It is not hard to obtain that $\sigma$ satisfies Lemma \ref{lem: +1case condition for central element}, so $x_1^2+x_2^2+y_1^2+y_2^2$ is a central element of $B$. In this case, the linear maps in the twisting system of $M_2(C_{A^!}(z))$ are
\begin{equation*}
\theta^{(0)}=\begin{pmatrix}
		\id & 0\\
		0& \xi_{-1}
	\end{pmatrix},\quad
	\theta^{(1)}=\begin{pmatrix}
		\sigma_{11}^! & \sigma_{12}^!\\
		\sigma^!_{21}& \sigma^!_{22}
	\end{pmatrix}\circ\xi_{-1},
\end{equation*}
and two liner transformations of $C_{A^!}(z)$ are
$$
	\Xi_1=2^{-1}(\id+\xi_{-1}), \qquad \Xi_2=2^{-1}(\id-\xi_{-1}).
$$
So the $\Z_2$-graded algebra $S$ equals $C_{A^!}(z)_0$ concentrated in degree $0$ and the $\Z_2$-graded $S$-module $M$ equals $C_{A^!}(z)_1$ concentrated in degree $1$. The other linear transformations are ${\Phi_1}_{\mid S}=\id_S$ and $\Phi_2$ mapping
$$	x^*_1\mapsto -2h(x_1^*-\sqrt{-1}x_2^*),\qquad x^*_2\mapsto 2\sqrt{-1}h(x_1^*+\sqrt{-1}x_2^*).$$

Then the semi-trivial extension $
\Lambda=S\ltimes_{\psi}{M(1)}$ is a $\Z_2\times\Z_2$-graded algebra where $\Lambda_{(0,0)}=C_{A^!}(z)_0$, $\Lambda_{(1,0)}=C_{A^!}(z)_1$ and  $\Lambda_{(0,1)}=\Lambda_{(1,1)}=0$ and
with multiplication
$$
(s,m)\odot(s',m')=(ss'+\Phi_2(m)m',sm'+ms'),
$$
for any $s,s'\in C_{A^!}(z)_0, m,m'\in C_{A^!}(z)_1$. Then  $\Lambda$ is a $\Z_2$-graded algebra concentrated in degree $0$ by forgetting the first grading and $\Lambda_0$ is also a $\Z_2$-graded algebra. In fact, one can find a $\Z_2$-graded algebra isomorphism between $\Lambda_0$ and $C_{A^!}(z)$. But we give another way to discuss here.

We view $\Lambda_0$ as a $\Z_2$-graded algebra. It is not hard to know that  ${\Phi_2}_{\mid M_1}$ is a linear automorphism of $M_1$ by Lemma \ref{lem: properties of Xi Phi}(5). We define a left Zhang-twisting system $\nu=\{\nu_0,\nu_1\}$ of $\Z_2$-graded algebra $C_{A^!}(z)$ by
$$
\nu_0=\id,\qquad {\nu_1}_{\mid C_{A^!}(z)_0}=\id, \qquad {\nu_1}_{\mid C_{A^!}(z)_1}={\Phi_2}_{\mid M_1}.
$$
So $\Lambda_0\cong {^{\nu} C_{A^!}(z)}$ as $\Z_2$-graded algebras. By Theorem \ref{thm: +1case equivalence to Lambda}, there are  equivalences of triangulated categories
	$$\umcm(B/(z+y_1^2+y_2^2)) \cong D^b(\gr_{\Z_2} \Lambda)\cong  D^b(\module \Lambda_0)^{\times 2}\cong D^b(\module {^{\nu} C_{A^!}(z)})^{\times 2}.$$
Moreover, $C_{A^!}(z)\cong \kk\Z_2\times \kk\Z_2$ as $\Z_2$-graded algebras, so the algebra ${^{\nu} C_{A^!}(z)}$ is $\Z_2$-graded semisimple by \cite[Corollary 4.4]{Z} and ${^{\nu} C_{A^!}(z)}\cong \kk\Z_2\times \kk\Z_2\cong C_{A^!}(z)$ since ${^{\nu} C_{A^!}(z)}$ is commutative by an easy computation. It implies that $B/(z+y_1^2+y_2^2)$ is a noncommutative graded isolated singularity and 
there is an equivalence of triangulated categories
	$$\umcm(B/(z+y_1^2+y_2^2)) \cong D^b(\module\kk)^{\times 8}.$$
\end{example}

\section{A Skew version of Kn\"orrer's Periodicity Theorem with $p_{12}=-1$}

Let $B=A_{\{-1,p_{11}\}}[y_1,y_2;\sigma]$ be a $\Z$-graded double Ore extension of a noetherian Koszul Artin-Schelter regular algebra $A=T(V)/(R)$  with $\deg y_1=\deg y_2=1$. In this case, $p_{11}$ can be nonzero. We will always assume $p_{11}=0$ except the following special values of $p_{11}$.

\begin{proposition}\label{prop: -1case p11=2genhao-1 not isolated singularity}
	Let $B=A_{\{-1,p_{11}\}}[y_1,y_2;\sigma]$ be a $\Z$-graded double Ore extension of a noetherian Koszul Artin-Schelter regular algebra $A$ and $z\in A_2$ be a nonzero regular central element of $A$.	Suppose $B$ is noetherian and $z+y_1^2+y_2^2$ is a regular central element of $B$. If  $p_{11}=\pm2\sqrt{-1}$, then $B/(z+y_1^2+y_2^2)$ is not a noncommutative  graded isolated singularity.
\end{proposition}
	\begin{proof}
		Write 
		$$\sigma'=
		\begin{pmatrix}
			\sigma_{11}+\cfrac{p_{11}}{2}\sigma_{12} &\sigma_{12}\\
			\sigma_{21}+\cfrac{p_{11}}{2}\sigma_{22}-\cfrac{p_{11}}{2}\sigma_{11}-\cfrac{p_{11}^2}{4}\sigma_{12} & \sigma_{22}-\cfrac{p_{11}}{2}\sigma_{12}
		\end{pmatrix}:A\to M_2(A),
		$$
		and
		$$\varphi'=
		\begin{pmatrix}
			\varphi_{11}+\cfrac{p_{11}}{2}\varphi_{21} &\varphi_{12}+\cfrac{p_{11}}{2}\varphi_{22}-\cfrac{p_{11}}{2}\varphi_{11}-\cfrac{p_{11}^2}{4}\varphi_{21}\\
			\varphi_{21}& \varphi_{22}-\cfrac{p_{11}}{2}\varphi_{21}
		\end{pmatrix}:A\to M_2(A),
		$$
		where $\varphi=(\varphi_{ij})$ is the inverse of $\sigma$.
		
		It is routine to check that $\sigma'$ is an invertible algebra homomorphism with the inverse $\varphi'$ and $A_{\{-1,0\}}[y_1,y_2;\sigma']$ is a $\Z$-graded double Ore extension by Lemma \ref{lem: invertible iff double ore extension}. Then the algebra homomorphism from $B$ to $A_{\{-1,0\}}[y_1,y_2;\sigma']$ mapping
		$$
		y_1\mapsto  y_1,\quad
		y_2\mapsto  y_2+\cfrac{p_{11}}{2}y_1,\quad a\mapsto a,\ \forall a\in A,
		$$
		is an isomorphism, and the image of $z+y_1^2+y_2^2$ is $z+y_2^2$ since $p_{11}=\pm2\sqrt{-1}$. So 
			$$B/(z+y_1^2+y_2^2)\cong A_{\{-1,0\}}[y_1,y_2;\sigma']/(z+y_2^2).
			$$ 

It is well known $k_{-1}[y_1,y_2]/(y_2^2)$ is not a noncommutative graded isolated singularity, and so 
$$\gldim_{\Z_2} C_{k_{-1}[y_1,y_2]^!}(y_2^2)>0$$ 
by Theorem \ref{thm: properties of clifford def}(2). Write $D=A_{\{-1,0\}}[y_1,y_2;\sigma']$. By Lemma \ref{lem: rep for Clifford deformation of double ore extensions}(4), $$\gldim_{\Z_2}C_{D^!}(z+y_2^2)\geq\gldim_{\Z_2} C_{k_{-1}[y_1,y_2]^!}(y_2^2)>0.$$ The result holds by Theorem \ref{thm: properties of clifford def}(2).
\end{proof}

\begin{lemma}\label{lem: two class of -1 case}
Let $B=A_{\{-1,p_{11}\}}[y_1,y_2;\sigma]$ be a $\Z$-graded double Ore extension of $A$ where $p_{11}\neq \pm2\sqrt{-1}$. Suppose the element $y^2_1+y^2_2$ is a central element in $B$, then there exists a $\Z$-graded double Ore extension $A_{\{-1,0\}}[y_1,y_2;\sigma']$ such that 
		$$B/(a+y_1^2+y_2^2)\cong 
			A_{\{-1,0\}}[y_1,y_2;\sigma']/(a+y_1^2+y_2^2), 
		$$
		for any $a\in A$.
\end{lemma}
\begin{proof}
 Write
	$$
	\sigma'=
	\begin{pmatrix}
			\sigma_{11}+\cfrac{p_{11}}{2}\sigma_{12} &c\sigma_{12}\\
			c^{-1}\left(\sigma_{21}+\cfrac{p_{11}}{2}\sigma_{22}-\cfrac{p_{11}}{2}\sigma_{11}-\cfrac{p_{11}^2}{4}\sigma_{12}\right) & \sigma_{22}-\cfrac{p_{11}}{2}\sigma_{12}
		\end{pmatrix}:A\to M_2(A),
	$$
	for some $c\in\kk$ satisfying $c^2=1+\frac{p_{11}^2}{4}$. 	It is similar to the proof of Proposition \ref{prop: -1case p11=2genhao-1 not isolated singularity} to obtain that  $A_{\{-1,0\}}[y_1,y_2;\sigma']$ is a $\Z$-graded double Ore extension. There is a $\Z$-graded algebra isomorphism $\rho$ from $B$ to $A_{\{-1,0\}}[y_1,y_2;\sigma']$ mapping
	$$
	y_1\mapsto  c^{-1}y_1,\quad
	y_2\mapsto  y_2+\cfrac{p_{11}}{2}c^{-1}y_1,\quad x\mapsto x,\ \forall x\in A.
	$$
Since $p_{11}\neq\pm2\sqrt{-1}$, then $\rho(a+y_1^2+y_2^2)=a+y_1^2+y_2^2$. 
\end{proof}

In the rest of this section, $B=A_{\{-1,0\}}[y_1,y_2;\sigma]$ be a $\Z$-graded double Ore extension of a noetherian Koszul Artin-Schelter regular algebra $A=T(V)/(R)$  with $\deg y_1=\deg y_2=1$, and  $z\in A_2$ be a nonzero regular central element of $A$. There are results similar to Lemma \ref{lem: +1case condition for central element} and  Corollary \ref{cor: properties for dual of sigma}.

\begin{lemma}\label{lem: -1case condition for central element}
	The element $z+y_1^2+y_2^2$ is a regular central element of $B$ if and only if 
	\begin{enumerate}
		\item $\sigma_{11}^2+\sigma_{21}^2=\sigma_{12}^2+\sigma_{22}^2=\id_A$;
		\item 
		$\sigma_{11}\sigma_{12}+\sigma_{21}\sigma_{22}=\sigma_{12}\sigma_{11}+\sigma_{22}\sigma_{21}$;
		\item $\sigma(z)=
		\begin{pmatrix}
			z &	0\\
			0	&	z
		\end{pmatrix}.
		$
	\end{enumerate}	
\end{lemma}

By Lemma \ref{lem: -1case condition for central element}(3), the construction \eqref{eq: dual map of simga from clifford deformation to its matrix}, \ref{lem: invertible iff double ore extension} and Lemma \ref{lem: inv of sigma and simga dual}, we have the following $t$-invertible algebra homomorphism
$$
\sigma^!=(\sigma_{ij}^!):C_{A^!}(z)\to M_2(C_{A^!}(z)).
$$

\begin{corollary}\label{cor: properties for dual of sigma case -1} If $z+y_1^2+y_2^2\in B$ is a regular central element, then
	\begin{enumerate}
		\item $(\sigma^!_{11})^2+(\sigma^!_{21})^2=(\sigma^!_{12})^2+(\sigma^!_{22})^2=\id_{C_{A^!}(z)};$
		\item 
		$\sigma^!_{12}\sigma^!_{11}+\sigma^!_{22}\sigma^!_{21}=\sigma^!_{11}\sigma^!_{12}+\sigma^!_{21}\sigma^!_{22};$
		\item $\sigma^!_{11}\sigma^!_{21}+\sigma^!_{21}\sigma^!_{11}=0;$
		\item $\sigma^!_{12}\sigma^!_{22}+\sigma^!_{22}\sigma^!_{12}=0;$
		\item $\sigma^!_{22}\sigma^!_{11}+\sigma^!_{12}\sigma^!_{21}=\sigma^!_{11}\sigma^!_{22}+\sigma^!_{21}\sigma^!_{12}.$
	\end{enumerate}	
\end{corollary}

Now we introduce a linear map
\begin{equation}\label{eq: def of theta of Eoplus E}
	\theta=(\theta_{ij})=\begin{pmatrix}
	\id & \sigma^!_{12}\sigma^!_{11}+\sigma^!_{22}\sigma^!_{21}\\
	0& \sigma^!_{22}\sigma^!_{11}+\sigma^!_{12}\sigma^!_{21}
\end{pmatrix}: C_{A^!}(z) \to M_2(C_{A^!}(z)),
\end{equation}
where each $\theta_{ij}$ is a $\Z_2$-graded map for any $i,j=1,2.$ Using the fact $\sigma^!$ is an algebra homomorphism, one obtains the following properites of $\theta.$

\begin{lemma} Retain the notations above. Then
	\begin{enumerate}		
		\item $\theta$ is an algebra homomorphism;
		\item $\theta_{12}\theta_{22}+\theta_{22}\theta_{12}=0$;
		
		\item $\theta_{22}^2+\theta_{12}^2=\id_{C_{A^!}(z)}$.
	\end{enumerate}
\end{lemma}

For the $\Z_2$-graded algebra $\kk\times \kk$, we choose an invertible $\kk$-basis  $\varepsilon=\{\varepsilon_1=(1,1),\varepsilon_2=(1,-1)\}$.

\begin{theorem}
	Let $\theta$ be the algebra homomorphisms from $C_{A^!}(z)$ to $M_2(C_{A^!}(z))$ defined in \eqref{eq: def of theta of Eoplus E}. Then $\theta_\times=(\theta,\varepsilon)$ is a twisting system of $C_{A^!}(z)\times C_{A^!}(z)$.
\end{theorem}
\begin{proof}
The linear map 
$$\begin{pmatrix}
	\id  & 0 \\-\varphi_{21}^!\varphi_{11}^!-\varphi_{22}^!\varphi_{12}^!&\varphi_{22}^!\varphi_{11}^!+\varphi_{21}^!\varphi_{12}^!
\end{pmatrix}:C_{A^!}(z)\to M_2(C_{A^!}(z))
$$
is the $t$-inverse of $\theta$. Clearly, $\theta(1)=I_2\in GL_2(\kk)$. 

In this case, $l_{1;jj'}=\delta_{jj'}$ and $l_{2;jj'}=\delta_{jj'}+1$ for any $j,j'=1,2.$ It is straightforward but tedious to check the following simplified conditions of Definition \ref{def: twisting system of direct product} hold
\begin{align*}
	\theta_{1j'}\left(\theta_{1j}(x)x'\right)+\theta_{2j'}\left(\theta_{2j}(x)x'\right)&=
		x\theta_{jj'}(x')+\theta_{12}(x)\theta_{uj'}(x'),\\
  	\theta_{1j'}\left(\theta_{2j}(x)x'\right)+\theta_{2j'}\left(\theta_{1j}(x)x'\right)&=\theta_{22}(x)\theta_{uj'}(x'),
\end{align*}
where $u\in \{1,2\}-j$ for any $j,j'=1,2$ and $x,x'\in C_{A^!}(z).$ 
\end{proof}

Hence, we have a $\Z_2$-graded algebra $\Gamma={^{\theta_\times} \left(C_{A^!}(z)\times C_{A^!}(z)\right)}$ by Corollary \ref{cor: twisted of E oplus E}, and the multiplication is as follows
\begin{equation}\label{eq: mul of twisted C_{A^!}(z) oplus C_{A^!}(z)}
\begin{array}{ll}
	\varepsilon_{1}x\star \varepsilon_{1}x'=\varepsilon_{1}xx',
	&	\varepsilon_{1}x\star\varepsilon_{2}x'=\varepsilon_{1}\theta_{12}(x)x'+\varepsilon_{2}\theta_{22}(x)x',
	\\
\varepsilon_{2}x\star\varepsilon_{1}x'=\varepsilon_{2}xx',
	&\varepsilon_{2}x\star \varepsilon_{2}x'=\varepsilon_{1}\theta_{22}(x)x'+\varepsilon_{2}\theta_{12}(x)x',
\end{array}
\end{equation}
for any $x,x'\in C_{A^!}(z)$.

\begin{lemma}\label{lem: definition of mu} The following map 
		$$
	\begin{array}{cclc}
		\mu: & \Gamma&\to &\Gamma\\
		~& \varepsilon_1a &\mapsto& \varepsilon_1\sigma^!_{11}\xi_{-1}(a)+\varepsilon_2\sigma^!_{21}\xi_{-1}(a),\\
		~& \varepsilon_2a &\mapsto& \varepsilon_1\sigma^!_{21}\xi_{-1}(a)+\varepsilon_2\sigma^!_{11}\xi_{-1}(a),
	\end{array}
	$$
is a $\Z_2$-graded algebra isomorphism satisfying $\mu^2={\id}$.
\end{lemma}
\begin{proof} It is straightforward to check that $\mu$ is an algebra homomorphism by \eqref{eq: mul of twisted C_{A^!}(z) oplus C_{A^!}(z)}. For any $a\in C_{A^!}(z)$, one obtains that
	\begin{align*}
		\mu^2(\varepsilon_1a)=\varepsilon_1\left((\sigma^!_{11})^2+(\sigma^!_{21})^2\right)(a)+\varepsilon_2\left(\sigma^!_{21}\sigma^!_{11}+\sigma^!_{11}\sigma^!_{21}\right)(a)=\varepsilon_1a,\\
		\mu^2(\varepsilon_2a)=\varepsilon_1\left(\sigma^!_{11}\sigma^!_{21}+\sigma^!_{21}\sigma^!_{11}\right)(a)+\varepsilon_2\left((\sigma^!_{21})^2+(\sigma^!_{11})^2\right)(a)=\varepsilon_2a.
	\end{align*}
Hence, $\mu^2=\id$ and $\mu$ is an isomorphism.
\end{proof}

By Example \ref{ex: twisted E oplus E}(2), the following map
$$
\begin{array}{cclc}
	\psi: & {_\mu\Gamma(1)}\otimes_{\Gamma} {_\mu\Gamma(1)} &\to &\Gamma\\
	~& \varepsilon_ia\otimes\varepsilon_{i'}a'&\mapsto &\mu(\varepsilon_ia)\varepsilon_{i'}a',
\end{array}
$$
is a $\Z_2$-graded $\Gamma$-bimodule homomorphism, and we have a semi-trivial extension
$$
\Gamma\ltimes_{\psi} ({_\mu\Gamma(1)}).
$$
By Remark \ref{rem: semi-trivial extension}, it is $\Z_2\times \Z_2$-graded and also $\Z_2$-graded by forgetting the first degree of $\Z_2\times \Z_2$-degrees. So viewing this semi-trivial extension as a $\Z_2$-graded algebra, the $0$-th homogeneous space is also a $\Z_2$-graded algebra.

\begin{lemma}\label{lem: -1case: semi-trivial extension is Zhang-twist} Retain the notations above. Then
\begin{enumerate}
    \item $\nu=\{\nu_0=\id,\nu_1=\mu\}$ is a left Zhang-twisting system of $\Z_2$-graded algebra $\Gamma$;
    \item ${^\nu \Gamma}\cong \left(\Gamma\ltimes_{\psi} ({_\mu\Gamma(1)})\right)_0$ as $\Z_2$-graded algebras.
\end{enumerate}
\end{lemma}
\begin{proof}
    Straightforwardly.
\end{proof}

\begin{theorem}\label{thm: -1case equivalence ot semi-trivial extension of Gamma}
	Let $B=A_{\{-1,0\}}[y_1,y_2;\sigma]$ be a $\Z$-graded double Ore extension of a noetherian Koszul Artin-Schelter regular algebra $A$ and $z\in A_2$ be a nonzero regular central element of $A$. Suppose $B$ is noetherian and  $z+y_1^2+y_2^2$ is a central element of $B$. Then 
 \begin{enumerate}
     \item $C_{B^!}(z+y_1^2+y_2^2)\cong\Gamma\ltimes_{\psi} ({_\mu\Gamma(1)})$ as $\Z_2$-graded algebras;
     \item  there are equivalences of triangulated categories
	$$\umcm(B/(z+y_1^2+y_2^2))\cong D^b\left(\gr_{\Z_{2}}
 \Gamma\ltimes_{\psi} ({_\mu\Gamma(1)})\right)\cong D^b\left(\module 
 {^\nu \Gamma}\right),$$
where $\nu=\{\nu_0=\id,\nu_1=\mu\}$ is a left Zhang-twisting system of $\Z_2$-graded algebra $\Gamma$.
 \end{enumerate}

\end{theorem}
\begin{proof}
(1) By Lemma \ref{lem: rep for Clifford deformation of double ore extensions}(1), $C_{B^!}(z+y_1^2+y_2^2)$ is generated by $V^*,y^*_1,y^*_2$ subject to the relations consisting of the relations of $C_{A^!}(z)$ and
\begin{align*}
	&(y_1^*)^2-1,(y_2^*)^2-1,y_1^*y_2^*-y_2^*y_1^*,\\ &v^*y_i^*+y_1^*\sigma_{1i}^*(v^*)+y_2^*\sigma_{2i}^*(v^*), \qquad \forall v^*\in V^*,i=1,2.
\end{align*}
Comparing with the multiplication list of $\Gamma\ltimes_{\psi} ({_\mu\Gamma(1)})$ (see \eqref{eq: mul of twisted C_{A^!}(z) oplus C_{A^!}(z)}), one obtains that 
	the $\Z_2$-graded algebra homomorphism from $C_{B^!}(z+y_1^2+y_2^2)$ to $\Gamma\ltimes_{\psi} ({_\mu\Gamma(1)})$ mapping
	$$y^*_1\mapsto (0,\varepsilon_1),\qquad y^*_2\mapsto (0,\varepsilon_2),\qquad v^*\mapsto (\varepsilon_1v^*,0),\ \forall v^*\in V^*,$$
	is an isomorphism. So $\Gamma\ltimes_{\psi} ({_\mu\Gamma(1)})$ is strongly $\Z_2$-graded by Theorem \ref{thm: properties of clifford def}(1).

(2) follows by Theorem \ref{thm: properties of clifford def} and Lemma \ref{lem: -1case: semi-trivial extension is Zhang-twist}.
\end{proof}

\begin{example}\label{ex: -1case}
Let $A$ be a noetherian Koszul Artin-Schelter regular algebra. Let $\sigma_1$ and $\sigma_2$ be graded automorphisms of $A$ satisfying
$$
\sigma_1^2=\sigma_2^2=\id_A,\qquad\sigma_{1}\sigma_2=\sigma_2\sigma_1.
$$
We extend $\sigma_2$ to be a graded automorphism $\widetilde{\sigma_2}$ of $A[y_1;\sigma_1]$ such that ${\widetilde{\sigma_2}}_{\mid A}=\sigma_2$ and  $\widetilde{\sigma_2}(y_1)=-y_1$. Then we have an Ore extension $A[y_1;\sigma_1][y_2;\widetilde{\sigma_2}]$ is isomorphic to the double Ore extension  $B=A_{\{-1,0\}}[y_1,y_2;\sigma]$, where
$$
\sigma=
\begin{pmatrix}
	\sigma_1 & 0\\
	0 &\sigma_2
\end{pmatrix}
:A\to M_2(A).
$$

Let $z\in A_2$ be a nonzero regular central element of $A$ such that $\sigma_1(z)=\sigma_2(z)=z$. By Lemma \ref{lem: -1case condition for central element}, $z+y_1^2+y_2^2$ is a central element of $B$. In this case, the linear map $\theta$ in the twisting system  $\theta_\times$ is
\begin{equation*}
	\theta=\begin{pmatrix}
		\id & 0 \\
		0& \sigma^!_{2}\sigma^!_{1}
	\end{pmatrix}: C_{A^!}(z) \to M_2(C_{A^!}(z)).
\end{equation*}
Suppose $\sigma_1=\sigma_2=\sigma$, $\theta$ is just  the diagonal map. So $\Gamma={^{\theta_\times} (C_{A^!}(z)\times C_{A^!}(z))}=C_{A^!}(z)\times C_{A^!}(z)$, and one obtains that $\bar{\nu}=\{\bar{\nu}_0=\id,\bar{\nu}_1=\sigma^{!}\xi_{-1}\}$ is a left Zhang-twisting system of $\Z_2$-graded algebra $C_{A^!}(z)$ such that 
 $$
 {^\nu \Gamma}\cong {^{\bar{\nu}} C_{A^!}(z)}\times {^{\bar{\nu}} C_{A^!}(z)},
 $$ 
as $\Z_2$-graded algebras, where $\nu=\{\nu_0=\id,\nu_1=\mu\}$ is a left Zhang-twisting system of $\Z_2$-graded algebra $\Gamma$. By Theorem \ref{thm: -1case equivalence ot semi-trivial extension of Gamma}, there is an   equivalence of triangulated categories
 	$$\umcm(B/(z+y_1^2+y_2^2)) \cong D^b\left(\module
  {^\nu C_{A^!}(z)}\right)^{\times 2}.$$
In particular, if $\sigma=\xi_{-1}$, ${^{\bar{\nu}} C_{A^!}(z)}= C_{A^!}(z)$ and so $\umcm(B/(z+y_1^2+y_2^2))\cong D^b\left(\module
 {C_{A^!}(z)}\right)^{\times 2}.$ It is exactly \cite[Theorem 1.4]{MU2} (also see Example \ref{ex: +1case two setp Ore extensions}).
\end{example}

\begin{example}
Let $A=k_{-1}[x_1,x_2]$ and the central element $z=x_1^2+x_2^2$. Let $B=A_{\{-1,0\}}[y_1,y_2;\sigma]$ be a $\Z$-graded double Ore extension of $A$ where $\sigma=(\sigma_{ij}):A\to M_2(A)$ satisfies
$$
\begin{array}{llll}
    \sigma_{11}(x_1)=-hx_1+hx_2, &  \sigma_{12}(x_1)=hx_1+hx_2, & \sigma_{21}(x_1)=hx_1+hx_2,& \sigma_{22}(x_1)=hx_1-hx_2, \\
        \sigma_{11}(x_2)=hx_1-hx_2, &  \sigma_{12}(x_2)=hx_1+hx_2, &  
    \sigma_{21}(x_2)=hx_1+hx_2,& \sigma_{22}(x_2)=-hx_1+hx_2,
\end{array}
$$
where $4h^2=1$. It is a special case of $\mathbb{T}$ in \cite{ZZ2}. One obtains that  $B$ is noetherian by \cite[Proposition 5.16]{ZZ2} and $\sigma$ satisfies Lemma \ref{lem: -1case condition for central element}. So $x_1^2+x_2^2+y_1^2+y_2^2$ is a central element of $B$. In this case, the linear map $\theta$ in the twisting system  $\theta_\times$ is
\begin{equation*}
\theta=(\theta_{ij})=\begin{pmatrix}
	\id & 0\\
	0& \theta_{22}
\end{pmatrix}: C_{A^!}(z) \to M_2(C_{A^!}(z)),
\end{equation*}
where $\theta_{22}(x_1^*)=x_2^*$, $\theta_{22}(x_2^*)=x_1^*$ and $\theta_{22}(x_1^*x_2^*)=x_1^*x_2^*$. Let $\Gamma={^{\theta_\times} (C_{A^!}(z)\times C_{A^!}(z))}$ be the twisted direct product, $\mu$ be the $\Z_2$-graded algebra isomorphism defined in Lemma \ref{lem: definition of mu} and ${^\nu \Gamma}$ be a Zhang-twisted algebra where $\nu=\{\nu_0=\id,\nu_1=\mu\}$ is a left Zhang-twisting system of $\Z_2$-graded algebra $\Gamma$.

Write $T_1,T_2,T_3,T_4$ and $T_5$ for right ${^\nu\Gamma}$-modules generated by $\{(1-x_1^*x_2^*,0),(x_1^*-x_2^*,0)\}$,  $\{\sqrt{-2h}(1+x_1^*x_2^*,0)+(x_1^*+x_2^*,0)\}$, $\{\sqrt{-2h}(1+x_1^*x_2^*,0)-(x_1^*+x_2^*,0)\}$, $\{\sqrt{2h}(0,1+x_1^*x_2^*)+(0,x_1^*+x_2^*)\}$ and $\{\sqrt{2h}(0,1+x_1^*x_2^*)-(0,x_1^*+x_2^*)\}$ respectively. It is routine to obtain that $T_1,T_2,T_3,T_4,T_5$ are all simple and non-isomorphic to each other, and
$
{^\nu\Gamma}\cong T_1^{\oplus 2}\oplus T_2\oplus T_3\oplus T_4\oplus T_5$  as right ${^\nu\Gamma}$-modules. So ${^\nu\Gamma}$ is semisimple and 
$$
{^\nu\Gamma}\cong M_2(\kk)\times \kk^{\times 4}.
$$
By Theorem \ref{thm: -1case equivalence ot semi-trivial extension of Gamma}, there is an   equivalence of triangulated categories
 	$$\umcm(B/(z+y_1^2+y_2^2)) \cong D^b(\kk)^{\times 5}.$$

\end{example}

Finally, we use a special case of the class $\mathbb{R}$ constructed in \cite{ZZ2} to show that noncommutative quadric hypersurfaces obtained by double Ore extension does not preserve the property of noncommutative graded isolated singularity. This behavior does not occur in the usual (second) double branched covers and tensor products of noncommutative quadric hypersurfaces.

\begin{proposition}\label{prop: double ore extension not preserve isolated singulariy}
Let $A=k_{-1}[x_1,x_2]$ and $B=A_{\{-1,0\}}[y_1,y_2;\sigma]$ be a $\Z$-graded double Ore extension of $A$ where $\sigma=(\sigma_{ij}):A\to M_2(A)$ satisfies
$$
\begin{array}{llll}
    \sigma_{11}(x_1)=x_1+x_2, &  \sigma_{12}(x_1)=x_1, & \sigma_{21}(x_1)=x_2,& \sigma_{22}(x_1)=0, \\
        \sigma_{11}(x_2)=0, &  \sigma_{12}(x_2)=x_1, &  
    \sigma_{21}(x_2)=-x_2,& \sigma_{22}(x_2)=-x_1+x_2.
\end{array}
$$
Then $B/(x_1^2+x_2^2+y_1^2+y_2^2)$ is not a noncommutative graded isolated singularity.
\end{proposition}
\begin{proof}
By  \cite[Proposition 0.5]{ZZ1}, $B$ is noetherian. It is straightforward to check that $\sigma$ satisfies Lemma \ref{lem: -1case condition for central element}, and so $x_1^2+x_2^2+y_1^2+y_2^2$ is a central element of $B$. So the noncommutative quadric hypersurface $B/(x_1^2+x_2^2+y_1^2+y_2^2)$ is well defined. Write $z=x_1^2+x_2^2$.

Note that $C_{A^!}(z)$ is generated by $x_1^*,x_2^*$ subject to relations
$
x_1^*x_2^*-x_2^*x_1^*, (x_1^*)^2-1, (x_2^*)^2-1.
$
Write $\varepsilon=\{\varepsilon_1=(1,1), \varepsilon_2=(1,-1)\}$. The linear map $\theta$ in the twisting system $\theta_\times=(\theta,\varepsilon)$ of $C_{A^!}(z)\times C_{A^!}(z)$ defined in \eqref{eq: def of theta of Eoplus E} is
\begin{equation*}
\theta=(\theta_{ij})=\begin{pmatrix}
	\id & \theta_{12}\\
	0& \theta_{22}
\end{pmatrix}: C_{A^!}(z) \to M_2(C_{A^!}(z)),
\end{equation*}
where 
$$
\begin{array}{lll}
\theta_{12}(x_1^*)=x_1^*+x_2^*, &\theta_{12}(x_2^*)=x_1^*-x_2^*,&\theta_{12}(x_1^*x_2^*)=2,\\
\theta_{22}(x_1^*)=-x_2^*, &\theta_{12}(x_2^*)=x_1^*,&\theta_{22}(x_1^*x_2^*)=-x_1^*x_2^*.
\end{array}
$$
Let $\Gamma={^{\theta_\times} (C_{A^!}(z)\times C_{A^!}(z))}$ be the twisted direct product,  $\mu$ be the $\Z_2$-graded algebra isomorphism defined in Lemma \ref{lem: definition of mu} and  ${^\nu \Gamma}$ be a Zhang-twisted algebra of $\Gamma$ where $\nu=\{\nu_0=\id,\nu_1=\mu\}$ is a left Zhang-twisting system of $\Z_2$-graded algebra $\Gamma$. We list partial multiplication of ${^\nu \Gamma}$:
$$  
\begin{array}{llll}
(1,0)\ast (a,b)=(a,0),&(x_1^*,0)\ast (0,b)=(0,0), \\
(x_1^*x_2^*,0)\ast (1,0)=(1,0), &
(x_1^*x_2^*,0)\ast (x_1^*x_2*,0)=(x_1^*x_2*,0),\\
(x_1^*x_2^*,0)\ast (x_1^*,0)=(x_1^*,0),&
(x_1^*x_2^*,0)\ast (x_2^*,0)=(x_2^*,0),\\
(x_1^*x_2^*,0)\ast (0,1)=(x_1^*x_2^*-1,0), &
(x_1^*x_2^*,0)\ast (0,x_1^*x_2*)=(1-x_1^*x_2*,0),\\
(x_1^*x_2^*,0)\ast (0,x_1^*)=(x_1^*-x_2^*,0),&
(x_1^*x_2^*,0)\ast (0,x_2^*)=(x_2^*-x_1^*,0),\\
(x_1^*,0)\ast (1,0)=(x_1^*,0),&
(x_1^*,0)\ast (x_1^*x_2^*,0)=(x_2^*,0),\\
(x_1^*,0)\ast (x_1^*,0)=(x_1^*,0), &
(x_1^*,0)\ast (x_2^*,0)=(x_2^*,0), \\
(x_2^*,0)\ast (1,0)=(x^*_1,0),&
(x_2^*,0)\ast (x_1^*x_2^*,0)=(x^*_2,0),\\
(x_2^*,0)\ast (x_1^*,0)=(-1,0), &
(x_2^*,0)\ast (x_2^*,0)=(-x_1^*x_2^*,0), \\
(x_2^*,0)\ast (0,1)=(x^*_2-x_1^*,0),&
(x_2^*,0)\ast (x_1^*x_2^*)=(x^*_1-x^*_2,0),\\
(x_2^*,0)\ast (0,x_1^*)=(x^*_1x^*_2-1,0), &
(x_2^*,0)\ast (0,x_2^*)=(1-x^*_1x^*_2,0), 
\end{array}
$$
for any $a,b\in C_{A^!}(z)$.

One obtains that the right ${^\nu\Gamma}$-module generated by $(1-x_1^*x_2^*,0)$ is nilpotent. So ${^\nu\Gamma}$ is not semisimple and $B/(x_1^2+x_2^2+y_1^2+y_2^2)$ is not a noncommutative graded isolated singularity by Lemma \ref{lem: -1case: semi-trivial extension is Zhang-twist}, Theorem \ref{thm: -1case equivalence ot semi-trivial extension of Gamma} and Theorem \ref{thm: properties of clifford def}.    
\end{proof}
\vskip7mm

\noindent {\bf Acknowledgments.} Yuan Shen is supported by Zhejiang Provincial Natural Science Foundation of China under Grant No. LY24A010006 and National Natural Science Foundation of China under Grant Nos. 11701515 and 12371101. Xin Wang is supported by Doctoral Research Fund of Shandong Jianzhu University under Grant No. XNBS1943.

\end{document}